\documentclass[11pt,reqno,palentino]{amsart} 

\usepackage{mypack}

\usepackage{latexsym}
\usepackage{a4wide}
\usepackage{amscd}
\usepackage{graphics}
\usepackage{amsmath}
\usepackage{amssymb}
\usepackage{mathrsfs}
\usepackage{accents}
\usepackage{enumerate}
\usepackage{soul,xcolor}
\usepackage{hyperref}
\usepackage{soul}

\newcommand{\fr}{\mathcal{F}} 
\newcommand{\frp}{\mathcal{F}^\mathrm{prop}} 
\newcommand{\ind}{{\rm ind}}

\DeclareMathOperator{\coker}{coker} 
\DeclareMathOperator{\id}{id}

\author[Alberto Abbondandolo and Thomas O. Rot]{Alberto Abbondandolo${}^1$ and Thomas O. Rot${}^2$}
\begin{document}

\setstcolor{red}

\title[Homotopy classes of proper Fredholm maps]{On the homotopy classification of proper Fredholm maps into a Hilbert space}
\maketitle 
\vspace{-.3cm}
\noindent\makebox[\textwidth][c]{
\begin{minipage}[t]{.4\textwidth}
\it \Small ${\ }^{1}$ Ruhr-Universit\"at Bochum\\
\phantom{${\ }^{1}$} Fakult\"at f\"ur Mathematik\\
\phantom{${\ }^{1}$} Universit\"atsstrasse 150\\
\phantom{${\ }^{1}$} Geb\"aude NA 4/33\\
\phantom{${\ }^{1}$} D-44801 Bochum, Germany\\
\phantom{${\ }^{1}$} Email: alberto.abbondandolo@rub.de
\end{minipage}%
\begin{minipage}[t]{.4\textwidth}
\it 
\Small ${\ }^{2}$ Vrije Universiteit Amsterdam\\
\phantom{${\ }^{2}$} Department of Mathematics\\
\phantom{${\ }^{2}$} De Boelelaan 1081\\
\phantom{${\ }^{2}$} 1081 HV, Amsterdam, the Netherlands\\
\phantom{${\ }^{2}$} Email: t.o.rot@vu.nl

\end{minipage}}

\vspace{.3cm}
\begin{abstract}
We classify the homotopy classes of proper Fredholm maps from an infinite dimensional Hilbert manifold into its model space in terms of a suitable version of framed cobordism. Our construction is an alternative approach to the classification introduced by Elworthy and Tromba in 1970 and does not make use of further structures on the ambient manifold, such as Fredholm structures.
In the special case of index zero, we obtain  a complete classification involving the Caccioppoli-Smale mod 2 degree and the absolute value of the oriented degree. 
\end{abstract}

\section*{Introduction}

In \cite{et70} Elworthy and Tromba classified the space of proper Fredholm maps from certain infinite dimensional Banach manifolds into their model Banach space modulo proper Fredholm homotopy, in the spirit of the Pontryagin framed cobordism theory. The first aim of this paper is to present an alternative approach to this classification, which we find to be simpler and which does not make use of additional structures on the infinite dimensional manifolds, such as Fredholm structures. The second aim is to specialize this classification to Fredholm maps of index zero and to relate our invariants to degree theory.

Unlike in \cite{et70}, which considers manifolds modeled on Banach spaces satisfying suitable assumptions, here we just treat the case of manifolds modeled on the separable infinite dimensional real Hilbert space $\mH$. In order to simplify the exposition, we deal only with smooth maps. In Appendix \ref{weakening}, we briefly discuss how our results can be generalized to more general Banach manifolds and to maps having finite regularity.

We now describe our approach and main results. Let $M$ and $N$ be paracompact smooth manifolds modeled on $\mH$. A smooth map $f:M \rightarrow N$ is said to be Fredholm if its differential at every point is a Fredholm operator. Since the general linear group $GL(\mH)$ of $\mH$ is contractible, the tangent bundle of any manifold modeled on $\mH$ is trivial. By fixing trivializations of $TM$ and $TN$, the differential of a Fredholm map $f:M \rightarrow N$ at any point can be seen as a Fredholm operator on $\mH$, and hence the section $df$ can be seen as a map
\[
df: M \rightarrow \Phi(\mH),
\]
where $\Phi(\mH)$ denotes the space of Fredholm operators on $\mH$. We will systematically see the differential of a Fredholm map $f:M \rightarrow N$ in this way. Notice that the fact that $GL(\mH)$ is contractible implies that the homotopy class of $df$ in $[M,\Phi(\mH)]$ does not depend on the chosen trivializations of $TM$ and $TN$.
The connected component of $\Phi(\mH)$ consisting of Fredholm operators of index $n\in \mZ$ is denoted by $\Phi_n(\mH)$. 

The map $f: M \rightarrow N$ is said to be proper if the inverse image of every compact subset of $N$ is compact. The question which we address in this paper is whether two proper Fredholm maps from $M$ to $\mH$ are proper Fredholm homotopic, that is, homotopic through a map $[0,1]\times M \rightarrow \mH$ which is also proper and Fredholm. 

As a warm up, we drop for the moment the properness requirement and we ask ourselves when two Fredholm maps are Fredholm homotopic. Here the codomain of the maps is allowed to be any Hilbert manifold $N$. This question can be reduced to the following extension problem: Let $U$ and $V$ be open subsets of $M$ with $\overline{U}\subset V$ and let $f: V \rightarrow N$ be a Fredholm map. When can we find a Fredholm map $\overline{f}: M \rightarrow N$ which coincides with $f$ on $U$? 

Two obvious necessary conditions are that $f|_U$ has a continuous extension from $M$ to $N$ and that $df|_{U}$ has a continuous extension from $M$ to $\Phi(\mH)$. It turns out that these conditions are also sufficient. See Theorem \ref{thm:extensionfredholm} below for the precise statement and Appendix A for its proof. This extension result implies that the classification problem of Fredholm maps up to Fredholm homotopy is quite simple: if we denote by $\mathcal{F}_n[M,N]$ the space of Fredholm maps of index $n$ from $M$ to $N$ modulo Fredholm homotopy, we have the following result:

\begin{maintheorem}
\label{main1}
Two smooth Fredholm maps $f,g:M\rightarrow N$ of index $n$ are Fredholm homotopic if and only if $f$ is homotopic to $g$ and $df: M \rightarrow \Phi_n(\mathbb{H})$ is homotopic to $dg: M \rightarrow \Phi_n(\mathbb{H})$. Moreover, the map $f \mapsto ([f],[df])$ induces a bijection 
\[
\mathcal{F}_n[M,N] \cong
[M,N] \times [M,\Phi_n(\mathbb{H})].
\]
\end{maintheorem}

This result is also contained in \cite[Proposition 2.24]{et70}, but for sake of completeness we prove it also here, as it follows quite easily from the above mentioned Fredholm extension theorem which is needed also later on.

When Fredholm maps and Fredholm homotopies are required to be proper, the classification question becomes more interesting. Here we give an answer to this question in the case $N=\mH$. We are considering the Hilbert $\mH$ as codomain of the maps, but we could replace it with the infinite dimensional sphere inside $\mH$, since this manifold is diffeomorphic to $\mH$, see \cite{bes66}.

The first step is to have an extension result for proper Fredholm maps. Let as before $U$ and $V$ be open subsets of $M$ with $\overline{U} \subset V$. Assume that the Fredholm map $f: V \rightarrow \mH$ is proper on some closed neighborhood of $\overline{U}$. The fact that $f$ has a continuous extension to $M$ is in this case automatic, because $\mH$ is contractible, but the assumption that the map $df: V \rightarrow \Phi(\mH)$ has a continuous extension to $M$ does not guarantee the existence of a proper Fredholm map $\overline{f}: M \rightarrow \mH$ which coincides with $f$ on $U$, as we show in Example \ref{nopropext} below. If we impose the further condition that $f(\partial U)$ is not the whole of $\mH$, then the existence of $\overline{f}$ as above can be proved. See Theorem \ref{thm:totalextension} for the precise statement. This extension result for proper Fredholm maps, which is a slight generalization of \cite[Lemma 4.2]{et70}, is our main technical tool for constructing maps and homotopies. It is proved in Appendix B.

Before describing our invariants for proper Fredholm maps into $\mH$, let us briefly review the finite-dimensional Pontryagin theory. Let $M$ be a closed $(k+n)$-dimensional manifold. The Pontryagin construction associates to a smooth map $f: M \rightarrow S^k$ the compact submanifold $X:=f^{-1}(y)$, where $y$ is a regular value of $f$, and endows it with the trivialization of its normal bundle which is induced by $df|_X$, seen as a section of the bundle
\[
\mathrm{hom}(TM|_X,\mR^k) \rightarrow X,
\]
after fixing an identification of $T_y S^k$ with $\mR^k$. This section satisfies $\ker df(x)=T_x X$, and one 
defines a framed submanifold of $M$ as a compact $n$-dimensional submanifold $X\subset M$ together with a section of the above bundle whose kernel at every $x\in X$ is $T_x X$. The framed submanifold $(f^{-1}(y),df|_{f^{-1}(y)})$, where $y$ is a regular value of $f$, is called the Pontryagin manifold of $f$ at $y$. 
Then one defines an equivalence relation on the set of framed submanifolds by using framed cobordisms and proves the following facts: 
\begin{enumerate}[(i)]
\item Pontryagin manifolds of a map at different regular points are framed cobordant.
\item Two maps are homotopic if and only if their Pontryagin manifolds are framed cobordant.
\item Every framed submanifold occurs as a Pontryagin manifold of some map $f:M \rightarrow S^k$.
\end{enumerate}
See \cite[Chapter 7]{mil65} for a complete account of the Pontryagin theory.

If $y\in \mH$ is a regular value of a proper Fredholm map $f: M \rightarrow \mH$ of index $n$, then $X:= f^{-1}(y)$ is a compact $n$-dimensional submanifold of $M$ and $df|_X$ can be seen as a map from $X$ into $\Phi_n(\mH)$, whose kernel at every $x\in X$ coincides with $T_x X$ (once again, we are using a global trivialization of $TM$). One is then tempted to define a framed submanifold of $M$ to be a compact $n$-dimensional submanifold $X\subset M$ together with a map $A:X \rightarrow \Phi_n(\mH)$ such that $\ker A(x)=T_x X$ for every $x\in X$. However, this definition has a disadvantage: the map $A$ does not necessarily 
extend to a map defined on the whole of $M$
, because the space $\Phi_n(\mH)$ has non-trivial topology when $\mH$ is infinite-dimensional. This means that not all pairs $(X,A)$ as above can occur as ``Pontryagin manifolds'' $(f^{-1}(y),df|_{f^{-1}(y)})$ of a proper Fredholm map $f: M \rightarrow \mH$.

Our way of overcoming this difficulty is to have the whole $M$ as domain of the framing: we define a framed submanifold of $M$ to be a pair $(X,A)$, where $X$ is a compact $n$-dimensional submanifold of $M$ and 
\[
A : M \rightarrow \Phi_n(\mH)
\]
is continuous, smooth on $X$ and such that $\ker A(x) = T_x X$ for every $x\in X$. The Pontryagin manifold of the proper Fredholm map $f: M \rightarrow \mH$ at the regular value $y\in \mH$ is then the pair $(f^{-1}(y),df)$. A framed cobordism between two framed submanifolds $(X_0,A_0)$ and $(X_1,A_1)$ is a pair $(W,B)$, where $W$ is a compact $(n+1)$-dimensional submanifold with boundary of $[0,1]\times M$, such that $\partial W = W \cap ( \{0,1\} \times M)$ and $W \cap (\{t\} \times M)$ equals $\{t\}\times X_0$ for $t$ close to $0$ and $\{t\} \times X_1$ for $t$ close to $1$, while
\[
B: [0,1]\times M\rightarrow \Phi_{n+1}(\mR\times\mH,\mH)
\]
is a continuous map, smooth on $W$, such that $\ker B(t,x) = T_{(t,x)} W$ for every $(t,x)\in W$ and $B(t,x) = 0 \oplus A_0(x)$ for all $t$ close to 0, $B(t,x) = 0 \oplus A_1(x)$ for all $t$ close to 1, for all $x\in M$. The Pontryagin manifold of $f$ changes by framed cobordism when we change the regular value.
Framed cobordism defines an equivalence relation on the set of framed submanifolds. The corresponding quotient space gives us a complete set of invariants for the homotopy classification of proper Fredholm maps. Indeed, we will prove the following theorem, where $\mathcal{F}_n^{\mathrm{prop}}[M,\mH]$ denotes the space of proper Fredholm maps of index $n$ from $M$ to $\mH$ modulo proper Fredholm homotopy:

\begin{maintheorem}
\label{main2}
Let $f,g:M\rightarrow \mH$ be proper Fredholm maps. Then $f$ is proper Fredholm homotopic to $
g$ if and only if the Pontryagin manifolds of $f$ and $g$ are framed cobordant. Moreover, the map which to each proper Fredholm map $f:M \rightarrow \mH$ of index $n$ associates the framed cobordism class of the Pontryagin manifold of a regular value induces a bijection from $\mathcal{F}_n^{\mathrm{prop}}[M,\mH]$ to the quotient space of framed $n$-dimensional compact submanifolds of $M$ modulo framed cobordism.
\end{maintheorem}

In the case of negative index $n<0$, the set $X$ in the $n$-dimensional framed submanifold $(X,A)$ is necessarily empty (we treat the empty set as a submanifold of arbitrary dimension, including negative dimension), so a $n$-dimensional framed submanifold is just given by a continuous map $A: M \rightarrow \Phi_n(\mH)$ with no further restrictions. In this case, the above theorem has the following immediate consequence:

\begin{maintheorem}
\label{main3}
Let $n$ be a negative integer. Then two proper Fredholm maps $f,g: M \rightarrow \mH$ of index $n$ are proper Fredholm homotopic if and only if their differentials are homotopic as maps from $M$ into $\Phi_n(\mH)$. Moreover, the map $f\mapsto df$ induces a bijection
\[
\mathcal{F}_n^{\mathrm{prop}}[M,\mH] \cong [M,\Phi_n(\mH)].
\]
\end{maintheorem}

Our last result concerns the specialization of Theorem \ref{main2} to index zero, which is less immediate. In order to simplify the statement, we assume $M$ to be connected. Before stating the result, we need to recall the notion of orientability of Fredholm maps of index zero, as defined by Fitzpatrick, Pejsachowicz, and Rabier in \cite{fpr94}. The fundamental group of $\Phi_0(\mH)$ is isomorphic to $\mZ_2$, and one says that a map $A: M \rightarrow \Phi_0(\mH)$ is orientable if the homomorphism
\[
A_* : \pi_1(M) \rightarrow \pi_1(\Phi_0(\mH)) \cong \mZ_2
\]
is zero. The orientability of $A: M \rightarrow \Phi_0(\mH)$ depends only on its homotopy class. The set of homotopy classes $[M,\Phi_0(\mH)]$ can be partitioned into two subsets, one given by the homotopy classes of maps $A: M \rightarrow \Phi_0(\mH)$ which are orientable, which we denote by $[M,\Phi_0(\mH)]_{\mathrm{or}}$, and the other one given by homotopy classes of non-orientable maps, which we denote by $[M,\Phi_0(\mH)]_{\mathrm{no}}$:
\[
[M,\Phi_0(\mH)] = [M,\Phi_0(\mH)]_{\mathrm{or}} \sqcup [M,\Phi_0(\mH)]_{\mathrm{no}}.
\]
An index zero Fredholm map $f: M \rightarrow \mH$ is said to be orientable if its differential $df: M \rightarrow \Phi_0(\mH)$ is orientable. An explicit example of a non-orientable map is given at the end of this section. This notion does not depend on the choice of the trivialization of $TM$ used to see $df$ as a map from $M$ into $\Phi_0(\mH)$.

To an orientable proper Fredholm map $f: M \rightarrow \mH$ of index zero we can associate an absolute degree $|\mathrm{deg}|(f)$, which is a non-negative integer defined in the following way. Let $y$ be a regular value of $f$. Then we can define an equivalence relation on the finite set $f^{-1}(y)$ by declaring two points $x_0,x_1\in f^{-1}(y)$ to be equivalent if, given a path $\gamma:[0,1] \rightarrow M$ from $x_0$ to $x_1$, the path $df\circ \gamma$ is trivial in $\pi_1(\Phi_0(\mH),GL(\mH)) \cong \mZ_2$. Due to the orientability of $f$, the triviality of $df\circ \gamma$ does not depend on the choice of the path connecting $x_0$ to $x_1$, and this equivalence relation has at most two equivalence classes. The non-negative integer $|\mathrm{deg}|(f)$ is defined to be the absolute value of the difference of the cardinalities of the two (possibly empty) equivalence classes. 

The notation $|\mathrm{deg}|(f)$ comes from the fact that the choice of an orientation of the map $f$, which means declaring one of the two equivalence classes in $f^{-1}(y)$ to be positive and the other to be negative, allows one to associate an oriented degree $\deg (f)\in \mZ$ to $f$, and $|\mathrm{deg}|(f)$ is precisely the absolute value of this integer valued degree. The oriented degree is invariant only with respect to particular proper Fredholm homotopies, and therefore is not relevant for the question we are studying, but its absolute value is a proper Fredholm invariant. This invariant refines the Caccioppoli-Smale mod 2 degree $\deg_2(f)$, which is defined as the parity of the set $f^{-1}(y)$ for $y$ a regular value of $f$ and is defined for all proper Fredholm maps of index zero, regardless of their orientability. See \cite{fpr92}, \cite{fpr94} and \cite{zr15} for more information on the oriented degree, the absolute degree and their history. 

After these preliminaries, we can state the result which classifies homotopy classes of proper Fredholm maps of index zero into $\mH$:

\begin{maintheorem}
\label{main4}
Two proper Fredholm maps $f,g: M \rightarrow \mH$ of index 0 on the connected manifold $M$ are proper Fredholm homotopic if and only if the following conditions are satisfied:
\begin{enumerate}[(i)]
\item The maps $df,dg: M \rightarrow \Phi_0(\mH)$ are homotopic and in particular are both orientable or both non-orientable.
\item If $f$ and $g$ are orientable, then $|\mathrm{deg}|(f)=|\mathrm{deg}|(g)$. If $f$ and $g$ are non-orientable, then $\deg_2(f)=\deg_2(g)$.
\end{enumerate}
Moreover, the map
\[
f \mapsto \left\{ \begin{array}{ll} \bigl([df], |\mathrm{deg}|(f)) & \mbox{if $f$ is orientable} \\ \bigl([df], \deg_2(f)) & \mbox{if $f$ is  non-orientable} \end{array} \right.
\]
induces a bijection
\[
\mathcal{F}_0^{\mathrm{prop}}[M,\mH] \cong \Bigl( [M,\Phi_0(\mH)]_{\mathrm{or}} \times \mN_0 \Bigr)\sqcup \Bigl( [M,\Phi_0(\mH)]_{\mathrm{no}} \times \mZ_2\Bigr).
\]
\end{maintheorem}
Here $\mN_0:= \{0\} \cup \mN $ denotes the set of non-negative integers. Notice the analogy with the finite-dimensional Hopf theorem on the homotopy classification of maps from a closed manifold $M$ to a sphere of the same dimension: two such maps are homotopic if and only if they have the same oriented degree - if $M$ is orientable - or the same mod 2 degree - if $M$ is non-orientable, see e.g. \cite{mil65}[p.\ 51]. In our infinite dimensional setting, the oriented degree has to be replaced by its absolute value and one has to ask that the differentials of the two maps are homotopic. 

If $M$ is simply connected, or more generally if $\pi_1(M)$ has no non-trivial homomorphisms into $\mZ_2$, then all maps $A: M \rightarrow \Phi_0(\mH)$ are orientable, and the latter bijection becomes
\[
\mathcal{F}_0^{\mathrm{prop}}[M,\mH] \cong  [M,\Phi_0(\mH)] \times \mN_0, \qquad [f] \mapsto ([df],|\mathrm{deg}|(f)).
\]
If $M$ is contractible then $[M,\Phi_0(\mH)]$ consists of only one element and the absolute degree becomes a complete invariant for the homotopy classification of proper Fredholm maps $f:M \rightarrow \mH$ of index zero:
\[
\mathcal{F}_0^{\mathrm{prop}}[M,\mH] \cong  \mN_0, \qquad [f] \mapsto |\mathrm{deg}|(f).
\]
Two Hilbert manifolds are homotopy equivalent if and only if they are diffeomorphic, cf.~\cite{bur69,mou68}. This implies in particular that $M\cong\mH$. In this case, it is easy to exhibit representatives for proper Fredholm maps of index zero of any absolute degree. Indeed, for every $n\in \mN_0$ we can consider the following smooth proper map of index zero:
\[
f_n : \mC \times \mH \rightarrow \mC \times \mH, \quad f_n(z,x) = \begin{cases}(z^n,x)\quad &n>0,\\
(|z|^2,x) &n=0.
\end{cases}
\]
When $n=0$ this map is not surjective and hence $ |\mathrm{deg}| (f_0)=0$. When $n\geq 1$ the differential of this map at every $(z,x)\in (\mC \setminus \{0\}) \times \mH$ is an isomorphism, so the $n$ points in the inverse image of the regular value $(1,0)\in \mC\times \mH$ are all pairwise equivalent and hence $ |\mathrm{deg}| (f_n)=n$. After identifying $\mC\times \mH$ with $\mH$ by a linear isomorphism, Theorem \ref{main4} implies that every proper Fredholm map $f:\mH \rightarrow \mH$ of index zero is proper Fredholm homotopic to one of the maps $f_n$, $n\in \mN_0$. The map
\[
f_{-n} : \mC \times \mH \rightarrow \mC \times \mH, \quad f_{-n}(z,x) = (\overline{z}^n,x),
\]
with $n\in \mN$, also has absolute degree $n$, and hence is proper Fredholm homotopic to $f_n$. An explicit proper Fredholm homotopy between $f_{-n}$ and $f_n$ can be easily found by noticing that $f_{-n} = f_n \circ C$, where $C$ is the real linear isomorphism
\[ 
C: \mC \times \mH \rightarrow \mC \times \mH, \qquad C(z,x) = (\overline{z},x),
\]
which can be joined to the identity within $GL(\mC \times \mH)$, as this space is connected.

We conclude this introduction by discussing an explicit example of a non-orientable proper Fredholm map of index zero with non-zero Caccioppoli-Smale degree. Let $E(\gamma^1_1)$ be the total space of the tautological line bundle $\gamma^1_1$ over $\mathbb{R}P^1$: An element of $E(\gamma_1^1)$ is a pair $(\ell,p)\in \mR P^1\times \mR^2$ with $p\in \ell$, where $\mR P^1$ is seen as the space of unoriented lines through the origin in $\mR^2$. Define the map $f:E(\gamma_1^1)\times \mH\rightarrow \mR^2\times \mH$, by
$$
f((\ell,p),x)=(p,x).
$$
We claim that this is a non-orientable index zero proper Fredholm map with $\deg_2(f)=1$. The space $E(\gamma_1^1)$ is diffeomorphic to $[0,\pi]\times \mR/\sim$ where the equivalence relation identifies $(0,r)$ with $(\pi,-r)$. In these coordinates the map $f$ is given by
\[
f(\theta,r,x)=(r\sin \theta,r \cos \theta,x).
\]
In this representation it is clearly seen to be an index zero proper Fredholm map. 
The differential of this map is
\[
df(\theta,r,x)[u,v,w]=(ur\cos \theta+v \sin\theta, -ur\sin \theta+v \cos \theta,w).
\]
Note that $df(\theta,r,x)$ is an isomorphism for all $r\not=0$, and for $r=0$ it has a one-dimensional kernel. Consider the path $\gamma:[0,\pi]\rightarrow E(\gamma_1^1)\times \mH$ given by
\[
\gamma(t):=(t,2t-\pi,0).
\]
Note that $\gamma(0)=(0,-\pi,0)=(\pi,\pi,0)=\gamma(\pi)$ so this defines a loop. The loop $df\circ \gamma$ is non-contractible in the space of Fredholm operators of index zero on $\mR^2 \times \mH$. To see this we compute
\[
df(\gamma(t))=\left(\begin{array}{ccc} (2t - \pi) \cos t &\sin t&0\\(\pi-2t) \sin t& \cos t&0\\0&0&1\end{array}\right).
\]
For each $t$, the operator $df(\gamma(t))$ is invertible, except at $t=\pi/2$, and the path goes transversely through the stratum of index zero Fredholm operators with a one dimensional kernel. This implies that the map $f$ is non-orientable, see Section~\ref{sec:nonpositive}. Every non-zero value in $\mR^2\times \mH$ is regular and has one preimage hence $\deg_2(f)=1$. 

Producing a non-orientable proper Fredholm map of index zero with vanishing Caccioppoli-Smale degree is also simple: just consider the map
\[
g: E(\gamma_1^1) \times \mR \times \mH \rightarrow \mR^2 \times \mR \times \mH, \qquad g((\ell,p),s,x) := (p,s^2,x).
\]
Computations similar to the above ones show that this map has the required properties. Theorem~\ref{main4} then implies that any non-orientable proper Fredholm map $E(\gamma^1_1)\times \mH\rightarrow \mH$ of index zero is homotopic  (after the obvious identifications) to $f$ or $g$.

\section{Setup} 

Let $\mH$ denote the separable infinite-dimensional real Hilbert space. By a Hilbert manifold we mean a paracompact Hausdorff smooth manifold modeled on $\mH$. 
Since the general linear group $GL(\mH)$ is contractible, see \cite{kui65}, the tangent bundle of every Hilbert manifold is trivial. Let $f: M \rightarrow N$ be a $C^1$ map between Hilbert manifolds. By choosing trivializations of $TM$ and $TN$, we can see the differential of $f$ as a map 
\[
df:M\rightarrow L(\mH)
\]
into the space $L(\mH)$ of bounded linear operators on $\mH$. We shall tacitly fix trivializations of all Hilbert manifold we encounter and always see the differential of a map between Hilbert manifolds in this way. When introducing notions which depend on the differential of a map, we shall discuss their behavior with respect to a change in the choice of these trivializations.

We denote by $\Phi(\mH)$ the space of bounded linear Fredholm operators on $\mH$ and by $\Phi_n(\mH)$ the component of  $\Phi(\mH)$ consisting of operators of Fredholm index $n\in \mZ$. When we wish to consider two distinct Hilbert spaces $\mH_1$ und $\mH_2$ as domain and codomain of the Fredholm operators, we write $\Phi(\mH_1,\mH_2)$ and $\Phi_n(\mH_1,\mH_2)$.

The $C^1$-map $f:M\rightarrow N$ is Fredholm if $df(x)$ belongs to $\Phi(\mH)$  for all $x$ in $M$. The Fredholm index of $df(x)$  is locally constant. Throughout this article we work with Fredholm maps $f$ such that the Fredholm index of $df(x)$ is globally constant. The common value of the Fredholm index of $df(x)$ is called the Fredholm index of the map $f$ and is denoted by $\ind f$. Therefore, the differential of a Fredholm map $f:M \rightarrow N$ of index $n$ is  a map
\[
df: M \rightarrow \Phi_n(\mH).
\]
Changing the trivializations of $TM$ and $TN$ changes $df$ by left and right multiplication with maps $M\rightarrow GL(\mH)$ and $N\rightarrow GL(\mH)$. Since $GL(\mH)$ is contractible, these maps are homotopic to the constant mapping which maps every point of $M$ into the identity, and hence the homotopy class of $df$ in $[M,\Phi_n(\mH)]$ does not depend on the choice of the trivializations of $TM$ and $TN$.

Recall Smale's generalization \cite{sma65} of Sard's Theorem.

\begin{theorem} 
\label{smale}
Let $f:M\rightarrow N$ be a $C^q$ Fredholm map with
$q>\max\{\ind f,0\}$. Then the set of regular values of $f$ is residual in $N$.
\end{theorem}

The situation is even better for proper maps. 

\begin{corollary} 
Let $f:M\rightarrow N$ be a $C^q$ Fredholm map with
$q>\max\{\ind f,0\}$. If $f$ is proper then the set of its regular
values is open and dense in $N$.
\end{corollary}

\begin{proof} 
The density follows from Theorem \ref{smale}, due to the fact that $N$ is a Baire space, because it admits a complete metric. Since the set of surjective operators is open in the space of all bounded operators, the set of regular points of $f$ is open. Therefore, the set of critical points of $f$ is closed. Since $f$ is proper, it is in particular a closed map, and hence the set of critical values of $f$ is closed. We conclude that the set of regular values is open.
\end{proof}

In order to avoid tracking differentiability degrees, we assume that all maps henceforth are smooth unless specified. We refer to Appendix \ref{weakening} for a short discussion about how to deal with maps having less regularity. 

A  Fredholm homotopy between two Fredholm maps $f,g: M \rightarrow N$ of index $n$ is a smooth Fredholm map $h:[0,1]\times M \rightarrow N$ such that $h(0,\cdot)=f$ and $h(1,\cdot)=g$.  The map $h$ has necessarily index $n+1$. By composition with a smooth function $[0,1]\rightarrow [0,1]$ having value 0 near 0 and 1 near 1, a Fredholm homotopy between $f$ and $g$ can be made to agree with $f$ on $[0,\epsilon]\times M$ and with $g$ on $[1-\epsilon]\times M$, for some $\epsilon \in (0,1/2)$. The possibility of this modification implies that being Fredholm homotopic is an equivalence relation. The symbol $\fr[M,N]$ denotes the set of Fredholm homotopy classes of maps from $M$ to $N$.

A Fredholm homotopy $h:[0,1] \times M \rightarrow N$ between two proper Fredholm maps $f$ and $g$ which is a proper map is called a proper Fredholm homotopy, and $f$ and $g$ are said to be proper Fredholm homotopic. This implies that for every $t\in [0,1]$ the map $h(t,\cdot)$ is proper, but the latter condition is in general weaker. To see this explicitly, let $\omega:\mR\rightarrow\mR$ be a smooth function with compact support such that $\omega(0)=1$. Then the function $h:[0,1]\times \mR\rightarrow \mR$
$$
h(t,x)=\begin{cases}(1-\omega(x-\frac{1}{t}))x\qquad &t>0\\
x&t=0
\end{cases},
$$
is a smooth function and the map $x\mapsto h(t,x)$ is proper for every $t$ as it equals the identity outside a compact set. However $h$ is not a proper map as $h(1/n,n)=0$ but the sequence $(1/n,n)$ does not have a convergent subsequence. Under certain uniform estimates, families of proper maps are proper, see \cite{zr15}.

The symbol $\frp[M,N]$ denotes the set of proper Fredholm homotopy classes of maps from $M$ to $N$. If we feel the need to specify the index we denote this with a subscript, as in $\fr_n[M,N]$ and $\frp_n[M,N]$.

We shall also make use of the following standard smoothing result:

\begin{lemma}
\label{smoothing}
Let $U$ and $V$ be open subsets of the Hilbert manifold $M$ with $\overline{U}\subset V$ and let $N$ be an open subset of a Banach space $X$. Let $F:M \rightarrow N$ be a continuous map such that $F|_V$ is smooth. Then there exists a smooth map $\tilde{F}:M \rightarrow N$ such that $\tilde{F}|_U = F|_U$ and $\tilde{F}$ is homotopic to $F$ relative $U$.
\end{lemma}

\begin{proof}
By the paracompactness and Hilbert structure of $M$ and by the continuity of $F$ we can find a countable smooth partition of unity $\{\phi_n\}_{n\in \mN_0}$, $\mN_0 =\{0\} \cup \mN$, with the following properties: 
\begin{enumerate}[(i)]
\item $\phi_0=1$ on $U$;
\item $\mathrm{supp}\, \phi_0 \subset V$;
\item For every finite subset $J\subset \mN_0$ and every choice of a point $x_j$ in the support of $\phi_j$, $j\in J$, we have
\[
\bigcap_{j\in J} \mathrm{supp}\, \phi_j \neq \emptyset \qquad \Rightarrow \qquad \mathrm{conv}\, \{ F(x_j) \mid j\in J\} \subset N.
\]
\end{enumerate}
Fix a point $x_n$ in the support of $\phi_n$ for every $n\in \mN$. 
By property (iii), the map 
\[
\tilde{F}(x) := \phi_0(x) F(x) + \sum_{n\in \mN} \phi_n(x) F(x_n), \qquad \forall x\in M,
\]
takes values into $N$. By (ii) and the smoothness of $F$ on $V$, the map $\tilde{F}$ is smooth on $M$. By (i), it agrees with $F$ on $U$. Using again (iii), we deduce that also the map 
\[
H(t,x) := t \tilde{F}(x) + (1-t) F(x) = \bigl( t \phi_0(x) + 1-t\bigr) F(x) + t \sum_{n\in \mN} \phi_n(x) F(x_n), \qquad \forall (t,x) \in [0,1]\times M,
\]
takes values into $N$. The map $H$ is a homotopy between $F$ and $\tilde{F}$ relative $U$.
\end{proof}

We shall use the above lemma either to smooth continuous maps $A: M \rightarrow \Phi_n(\mH)$, using the fact that $\Phi_n(\mH)$ is an open subset of the space of bounded linear operators, or to smoothen continuous maps $f: M \rightarrow N$ into a Hilbert manifold, using the fact that every Hilbert manifold can be embedded as an open subset of $\mH$ by a well known result of Eells and Elworthy, see \cite{ee70}.

\section{Fredholm extensions}

Given a Fredholm map $f$ defined on a subset of $M$ to $N$ we wish to discuss the problem of finding a Fredholm map $\overline f:M\rightarrow N$ that extends $f$. There are two obvious homotopy theoretic necessary conditions for such an extension to exist. First of all the map $f$ needs to have a continuous $N$-valued extension, and secondly the map $df$ needs to have a continuous extension into $\Phi(\mH)$, as this space has non-trivial topology. The next result shows that these are the only obstructions. The proof is contained in Appendix~\ref{app:extension}.

\begin{theorem}[Extension of Fredholm maps]
\label{thm:extensionfredholm}
Let $M,N$  be Hilbert manifolds and $U,V$ be open subsets of $M$ with $\overline U\subset V$. Let $f:V \rightarrow N$ be a Fredholm map of index $n$ and assume that $f$ and $df$ satisfy:
\begin{itemize}
\item[(i)] there is a continuous map $g: M \rightarrow N$ such that $g\bigr|_V=f$;
\item[(ii)] there is a continuous map $A:M \rightarrow \Phi_n(\mH)$ such that $A\bigr|_V= df$.
\end{itemize}
Then there exists a Fredholm map $\overline{f}:M \rightarrow N$ of index $n$ such that $\overline{f}\bigr|_U=f\bigr|_U$ and:
\begin{itemize}
\item[(i')] the map $\overline f:M\rightarrow N$ is homotopic to $g$ relative $U$;
\item[(ii')] the map $d\overline f:M\rightarrow \Phi_n(\mH)$ is homotopic to $A$ relative $U$. 
\end{itemize}
\end{theorem}

This extension result has various interesting corollaries, the first of which is Theorem \ref{main1} from the Introduction:

\begin{corollary}
\label{cor:homotopyclassesfredholm}
Two smooth Fredholm maps $f_0,f_1:M\rightarrow N$ of index $n$ are Fredholm homotopic if and only if $f_0$ is homotopic to $f_1$ and $df_0: M \rightarrow \Phi_n(\mathbb{H})$ is homotopic to $df_1: M \rightarrow \Phi_n(\mathbb{H})$. Moreover, the map $f \mapsto ([f],[df])$ induces a bijection 
\[
\mathcal{F}_n[M,N] \cong
[M,N] \times [M,\Phi_n(\mathbb{H})].
\]

\end{corollary}

\begin{proof}
If $f_0$ and $f_1$ are homotopic through the smooth Fredholm map $h:[0,1] \times M \rightarrow N$ of index $n+1$, then $h$, resp.\ $dh$, is a homotopy between $f_0$ and $f_1$, resp.\ $df_0$ and $df_1$.

Conversely, assume that $g:[0,1]\times M \rightarrow N$ is a homotopy between $f_0$ and $f_1$ and $\tilde{A}:[0,1]\times M \rightarrow \Phi_n(\mH)$ is a homotopy between $df_0$ and $df_1$. 
By considering their trivial extensions, we may assume $g$ and $\tilde{A}$ are defined on the Hilbert manifold $\mR \times M$. By reparametrizing the homotopies, we can also assume that
\[
\begin{split}
g(t,x) &= \left\{ \begin{array}{ll} f_0(x) & \forall (t,x) \in (-\infty,1/3)\times M, \\ f_1(x) & \forall (t,x) \in (2/3,+\infty) \times M, \end{array} \right. \\
\tilde{A}(t,x) &= \left\{ \begin{array}{ll} df_0(x) & \forall (t,x) \in (-\infty,1/3)\times M, \\ df_1(x) & \forall (t,x) \in (2/3,+\infty) \times M. \end{array} \right.
\end{split}
\]
Up to a regularization by a smooth partition of unity, we can further assume that $g$ and $\tilde{A}$ are smooth (see Lemma \ref{smoothing}).  
We also define the smooth map
\[
A : \mR \times M \rightarrow \Phi_{n+1}(\mR\times \mH,\mH)
\]
as
\[
A(t,x)(s,u) = \tilde{A}(t,x) u \qquad \forall (t,x)\in \mR \times M, \; \forall (s,u) \in \mR \times \mH,
\]
so that
\[
dg(t,x) = A(t,x) \qquad \forall (t,x)\in \bigl( \mR \setminus [1/3,2/3]\bigr) \times M.
\]
We now apply Theorem \ref{thm:extensionfredholm} to the data
\[ 
\begin{split}
& U :=  \bigl( \mR \setminus [1/4,3/4]\bigr) \times M, \quad V:=  \bigl( \mR \setminus [1/3,2/3]\bigr) \times M,  \\ & g:\mR \times M \rightarrow N, \quad A: \mR\times M \rightarrow \Phi_{n+1}(\mR \times \mH,\mH). 
\end{split}
\]
We obtain a Fredholm map $h: \mR \times M \rightarrow N$ of index $n+1$ such that $h|_{U} = g|_{U}$. In particular, the restriction of $h$ to $[0,1]\times M$ is the required Fredholm homotopy between $f_0$ and $f_1$. This proves the first part of the statement.

Now consider the map
\[
\mathcal{F}[M,N] \rightarrow [M,N]\times [M,\Phi_n(\mH)]
\]
which is induced by $f \mapsto ([f],[df])$. By the statement that we just proved, this map is well defined and injective. In order to prove that this map is surjective, consider two arbitrary continuous maps $g: M \rightarrow N$ and $A: M \rightarrow \Phi_n(\mH)$. By applying Theorem  \ref{thm:extensionfredholm} to these maps together with $U=V=\emptyset$, we obtain a Fredholm map $\overline{f}: M \rightarrow N$ of index $n$ such that $\overline{f}$ is homotopic to $f$ and $d\overline{f}$ is homotopic to $A$. This proves the surjectivity. 
\end{proof}
\begin{corollary}
\label{cor:onehomotopyclass}
There is exactly one homotopy class of Fredholm maps $f:\mH\rightarrow \mH$ for each index $n$. 
\end{corollary}
\begin{proof}
Take $M=N=\mH$, which is contractible, and apply Corollary \ref{cor:homotopyclassesfredholm}. Hence $[M,N]$ has exactly one element, and $[M,\Phi_n(\mH)]$ equals the set of path components of $\Phi_n(\mH)$. This space is path connected, which proves the corollary.
\end{proof}

Theorem \ref{thm:extensionfredholm} is easier to apply in the case that $N=\mH$ and has the following consequence.

\begin{corollary}
\label{fredextH}
Let $M$ be a Hilbert manifold and $U,W$  be open subsets of $M$ such that $\overline U\subset W$. Let $A: M\rightarrow \Phi_n(\mH)$ be a continuous map and $f: W\rightarrow \mH$ a smooth map such that $df= A|_W$. Then there exists a Fredholm map $\overline f:M\rightarrow \mH$ such that $\overline f|_U = f|_U$ and $d\overline{f}: M \rightarrow \Phi_n(\mH)$ is homotopic to $A$ relative $U$.  
\end{corollary}

\begin{proof} Let $V$ be an open neighborhood of $\overline{U}$ such that $\overline{V} \subset W$.
Since $M$ is metrizable and $\mH$ is locally convex, Dugundji's generalization of Tietze's extension theorem \cite{dug51} provides us with a continuous map $g:M\rightarrow \mH$ such that $g\bigr|_{\overline{V}}=f\bigr|_{\overline{V}}$. The hypothesis of Theorem~\ref{thm:extensionfredholm} are satisfied by the data $M$, $N$, $U$, $V$, $f$, $g$, $A$, and from this theorem we immediately get the required map $\overline f$. 
\end{proof}

Corollary \ref{cor:onehomotopyclass} shows that the homotopy theory for Fredholm maps between Hilbert spaces is quite poor. However the situation becomes richer when we consider proper Fredholm maps.

\section{Proper Fredholm extensions into Hilbert spaces}

The aim of this section is to discuss the extension problem for proper Fredholm maps. For sake of simplicity, we just consider maps into the Hilbert space $\mH$. If in Corollary \ref{fredextH} we assume the Fredholm map $f: W \rightarrow \mH$ to be proper on a closed neighborhood of $\overline{U}$, we cannot in general require the extension $\overline{f}$ to be proper (see Example \ref{nopropext} below). However, this becomes true if we add the assumption that $f|_{\partial U}$ is not surjective. More precisely, we have the following result, whose proof follows closely the proof of \cite[Lemma 4.2]{et70} and is contained in Appendix~\ref{app:properextention} below.

\begin{theorem}[Extension of proper Fredholm maps]
\label{thm:totalextension}
Let $M$ be a manifold modeled on $\mH$, possibly with boundary.
Let $U\subset V\subset W$ be open subsets of $M$ such that $\overline U\subset V \subset \overline V \subset W$. Let $f:W\rightarrow \mH$ be a Fredholm map of index $n$. Suppose that:
\begin{enumerate}[(i)]
\item there exists a continuous map $A: M \rightarrow \Phi_n(\mH)$ such that $A|_W= df$; 
\item the map $f|_{\overline V}$ is proper;
\item there exists a point $z$ in $\mH \setminus f(\partial U)$.
\end{enumerate}
Then there exists a proper Fredholm map $\overline f:M\rightarrow \mH$ of index $n$ such that $\overline f|_U = f|_U$ and
\begin{enumerate}[(i')]
\item $d\overline f$ is homotopic to $A$ relative $U$;
\item there exists a neighborhood $Z\subset \mH$ of $z$ such that $\overline f^{-1}(Z)=f^{-1}(Z)\cap U$. 
\end{enumerate}
\end{theorem}

In the next example we show that we cannot drop assumption (iii) in the above theorem.

\begin{example} 
\label{nopropext}
Let $\chi: \mR \rightarrow [0,1]$ be a smooth function such that $\chi(t)=1$ for $t\leq 1/3$ and $\chi(t)=0$ for $t\geq 2/3$. Define $\varphi: \mR^2 \rightarrow \mR$ as
\[
\varphi(s,t) = \chi(t) s + \bigl(1-\chi(t)\bigr) s^2.
\]
The smooth function $\varphi$ is not proper on $\mR \times [0,1]$ because for each $s\leq 0$ the interval $[s,s^2]$ contains $0$ and hence there is some $t\in [0,1]$ such that $\varphi(s,t)=0$. However, $\varphi$ is proper on the set
\[
\mR \times \bigl( [-1,1/3] \cup [2/3,2] \bigr)
\]
because of the identity
\[
\varphi(s,t) = \left\{ \begin{array}{ll} s & \mbox{on } \mR \times  [-1,1/3], \\ s^2 & \mbox{on } \mR \times  [2/3,2], \end{array} \right.
\]
together with the properness of
the functions $s\mapsto s$ and $s\mapsto s^2$ on $\mR$ and compactness of the intervals $[-1,1/3]$ and $[2/3,2]$.

Now we consider the smooth map
\[
f: \mR \times \mR \times \mH \rightarrow \mR\times \mH, \qquad (s,t,x) \mapsto \bigl( \varphi(s,t),x \bigr).
\]
This map is Fredholm of index 1. By the properties of $\varphi$, $f$ is not proper on $\mR \times [0,1] \times \mH$, but it is proper on the closure of the open set
\[
V := \mR \times \bigl[ (-1,1/3) \cup (2/3,2) \bigr] \times \mH.
\]
As $U$, $W$ and $M$ we choose the following open subsets of $\mR \times \mR \times \mH \cong \mH$
\[
U:= \mR \times \bigl[ (-1/2,1/4) \cup (3/4,3/2) \bigr] \times \mH, \qquad W = M := \mR \times \mR\times \mH,
\]
so that
\[
\overline{U} \subset V \subset \overline{V} \subset W.
\]
The map $f: W \rightarrow \mR \times \mH\cong \mH$ fulfills the assumptions (i) and (ii) of Theorem \ref{thm:totalextension}, but not (iii) because
\[
f(\partial U) = f \bigl( \mR \times \{-1/2,1/4, 3/4,3/2\} \times \mH) \supset f \bigl( \mR \times \{1/4\} \times \mH) = \mR \times \mH.
\]
We claim that there is no proper Fredholm map $\overline{f} : \mR \times \mR \times \mH \rightarrow \mR \times \mH$ which coincides with $f$ on $U$. Indeed, such a map would give a proper Fredholm homotopy between the index-$0$ Fredholm maps $f_0,f_1:\mR \times \mH \rightarrow \mR \times \mH$ defined by
\[
f_0(s,x)=(s,x), \qquad f_1(s,x)=(s^2,x).
\]
However, these maps are not proper Fredholm homotopic. Recall that the Caccioppoli-Smale degree $\deg_2$, defined as number of points in the preimage of a regular value modulo 2, is a proper Fredholm homotopy invariant, see \cite{cac36,sma65}. The Caccioppoli-Smale degree of $f_0=\mathrm{id}$ is 1, and the Caccioppoli-Smale degree of $f_1$ is 0, as this map is not surjective. Hence $f_0$ and $f_1$ are not proper Fredholm homotopic and there is no proper Fredholm extension $\overline f$. 
\end{example}

\section{Framed cobordism is an invariant of proper homotopy classes of proper Fredholm maps.}

In the following definition, we consider the empty set as a manifold of any dimension $n\in \mZ$.

\begin{definition} 
\label{defi:framing}
Let $n\in \mZ$ and let $X\subset M$ be a compact submanifold of dimension $n$. A \emph{framing} of $X$ in $M$ is a continuous map $A:M\rightarrow \Phi_n(\mH)$ which is smooth on $X$ and such that for every $x\in X$ the sequence
\begin{equation}
\label{eq:framing}
0\rightarrow T_xX\rightarrow \mH \xrightarrow{A(x)} \mH \rightarrow 0
\end{equation}
is exact. The pair $(X,A)$ is called a {\em framed submanifold} of $M$.
\end{definition}

The trivialization of the tangent bundle of $M$ enters in the first map $T_x X \rightarrow \mH$ in (\ref{eq:framing}), which is given by composing the inclusion of $T_x X$ into $T_x M$ with the isomorphism $T_x M \cong \mH$ given by the trivialization.

When we refer to a framed submanifold $(X,A)$ of $M$, we always assume $X$ to be compact. Notice that, as discussed in the Introduction, a framing of the submanifold $X$ is always defined on the whole  of $M$. Notice also that in the case $n<0$ the submanifold $X$ must be empty. A framing of the empty set is an arbitrary continuous map $A:M\rightarrow \Phi_n(\mH)$. 

\begin{remark}
\label{rem:det}
A framed submanifold of a simply connected Hilbert manifold $M$ is automatically orientable. Indeed, let $(X,A)$ be an $n$-dimensional framed submanifold of $M$. The space of Fredholm operators $\Phi(\mH)$ is the base space of a real line bundle
\[
\det \rightarrow \Phi(\mH),
\]
which is called the determinant bundle and whose fiber at $T\in \Phi(\mH)$ is the 1-dimensional space
\[
\det (T) := \Lambda^{\max}(\ker T) \otimes \Lambda^{\max}(\coker T)^*,
\]
where $\Lambda^{\max}(V)$ denotes the top degree component in the exterior algebra of the finite dimensional real vector space $V$. This bundle was introduced by Quillen in \cite{qui85}; see also \cite{wan05,am09} for some of its properties. The determinant bundle pulls back by $A$ to a line bundle $A^*(\det)$ over $M$, which is trivial because $M$ is simply connected. Let $x\in X$. Since $A(x)$ is surjective and has kernel $T_x X$, the fiber of this line bundle at $x$ is 
\[
 \Lambda^n(\ker A(x))\otimes \Lambda^0(\coker A)^*=\Lambda^n(T_x X)\otimes \mR^*\cong  \Lambda^n(T_x X).
\]
Therefore, a global trivialization of $A^*(\det)$ induces by restriction a global trivialization of $\Lambda^n(TX)$ and hence an orientation of $X$. 
\end{remark}

\begin{remark}
\label{stablerem}
Actually, the tangent bundle $TX$ of an $n$-dimensional framed submanifold $(X,A)$ of $M$ is stably trivial if $\pi_k(M)=0$ for all $k\leq n$. Indeed, the Atiyah-J\"anich Theorem \cite[Appendix]{ati89} states that for a compact Hausdorff space $X$ the homotopy classes of maps from $X$ to $\Phi(\mH)$ are in one-to-one correspondence with the $K$-theory classes of real vector bundles of $X$. The isomorphism maps the homotopy class of the map $A\bigr|_X$, seen as a map to all Fredholm operators $\Phi(\mH)$, to the $K$-theory class of the tangent bundle $TX$. The map $A\bigr|_X$ factors as $X\rightarrow M\rightarrow \Phi(\mH)$. Since $\pi_k(M)$ vanishes for all $k\leq\dim X$, there is only one homotopy class of maps $X\rightarrow M$, see \cite[Corollary VII.13.16]{bre97}, so $A\bigr|_X$ is homotopic to a constant map. This in turn implies that the $K$-theory class $[TX]$ is equal to the $K$-theory class of the trivial $n$-dimensional vector bundle over $X$, which precisely means that $TX$ is stably trivial. 
\end{remark}

Let $f: M \rightarrow \mH$ be a proper Fredholm map of index $n$. Let $y\in \mH$ be a regular value of $f$. Then $f^{-1}(y)$ is a compact submanifold of $M$ of dimension $n$. By the regularity of $y$, the map $df:M \rightarrow \Phi_n(\mH)$ is a framing of $f^{-1}(y)$.

\begin{definition}
Let $y\in \mH$ be a regular value of the proper Fredholm map $f:M \rightarrow \mH$. The framed submanifold $(f^{-1}(y),df)$ is called the \emph{Pontryagin manifold} of $f$ at $y$.
\end{definition}

Notice that the Pontryagin manifold of $f$ at $y$ depends on the choice of the trivialization of $TM$. After a change of the trivialization of $TM$, the framing $df: M \rightarrow \Phi(\mH)$ gets multiplied on the right by a smooth map $G:M \rightarrow GL(\mH)$. 

\begin{remark}
Note that if $f^{-1}(y)$ intersects each component of $M$ then the framing determines the map $f$ completely. Indeed, we have
\[
f(x)=\int_{\gamma} df-y, 
\]
where $\gamma$ is a path from a point in $f^{-1}(y)$ to $x$. It turns out that for a framed submanifold $(X,A)$, $A$ is always homotopic relative $X$ to a Pontryagin manifold of some map. We could have defined the framing to be a homotopy class of framings as we define, but this is cumbersome to work with.
\end{remark}

\begin{definition}
A \emph{cobordism} between compact $n$-dimensional submanifolds $X_0,X_1\subset M$ is a compact submanifold with boundary $W\subset [0,1]\times M$ of dimension $n+1$ such that
\[
\partial W\subset \{0,1\} \times M,\quad \bigl( [0,\epsilon)  \times M\bigr)\cap W = [0,\epsilon) \times X_0, \quad \bigl( (1-\epsilon,1] \times M \bigr)\cap W = (1-\epsilon,1] \times X_1,
\]
for some $\epsilon>0$. A \emph{framing} of the cobordism $W$ is a continuous map  $B:[0,1]\times M\rightarrow \Phi_{n+1}(\mR\times \mH, \mH)$ which is smooth on $W$ and such that for every $(t,x)\in W$ the sequence
\[
0\rightarrow T_{(t,x)} W \rightarrow \mR\times \mH \xrightarrow{B(t,x)} \mH \rightarrow 0
\]
is exact, and for every $t_0\in[0,\epsilon)$, $t_1\in (1-\epsilon,1]$, $x\in M$,
\[
B(t_0,x)(s,u) = A_0(x)u, \qquad B(t_1,x)(s,u) = A_1(x)u \qquad \forall (s,u)\in \mR \times \mH,
\]
where $A_0,A_1: M \rightarrow \Phi_n(\mH)$ are framings of $X_0$ and $X_1$, respectively. In this case,
the pair $(W,B)$ is said to be a \emph{framed cobordism} between the framed submanifolds $(X_0,A_0)$ and $(X_1,A_1)$.
\end{definition}
In particular, the map
\[
[0,1]\times M \rightarrow \Phi_n(\{0\} \times \mH,\mH) \cong \Phi_n(\mH), \qquad (t,x) \mapsto B(t,x)|_{\{0\} \times \mH},
\]
defines a homotopy between $A_0$ and $A_1$.

The above definition is modeled after the following example: let $F: [0,1] \times M \rightarrow \mH$ be a smooth proper Fredholm homotopy between proper Fredholm maps $f,g: M \rightarrow \mH$ such that $F(t,x)=f(x)$ for $t\in[0,\epsilon)$ and $F(t,x)=g(x)$ for $t\in(1-\epsilon,1]$. Let $y\in \mH$ be a regular value for $F$; then it follows that $y$ is a regular value of $f$ and $g$.  The pair $(F^{-1}(y),dF)$ is a framed cobordism between the Pontryagin manifolds $(f^{-1}(y),df)$ and  $(g^{-1}(y),dg)$. 

The proof of the following fact is straightforward.

\begin{proposition}
Framed cobordism induces an equivalence relation on the set of framed submanifolds of $M$.
\end{proposition}
 
In the remaining part of this section we want to show that the framed cobordism class of the Pontryagin manifold of $f$ is an invariant of the proper Fredholm homotopy class of $f$. The argument is the same as in finite dimensional Pontryagin theory.

\begin{lemma}
\label{lem:indeppert} 
Let $f: M \rightarrow \mH$ be a proper Fredholm map and let $U\subset \mH$ be a convex open set consisting of regular values of $f$. Then $(f^{-1}(y),df)$ is framed cobordant to $(f^{-1}(z),df)$ for each $y,z\in U$.
\end{lemma}

\begin{proof} 
Let $\omega:[0,1]\rightarrow [0,1]$ be a smooth function such that $\omega=0$ on $[0,1/3]$ and $\omega=1$ on $[2/3,1]$. Let $F:[0,1]\times M \rightarrow \mH$ be the smooth homotopy 
\[
F(t,x):=f(x)-\omega(t)(z-y).
\]
The formula
\[
dF(t,x)(s,u) = df(x)u - s \, \omega'(t) (z-y), \qquad \forall (t,x)\in [0,1]\times M, \; \forall (s,u)\in \mR \times \mH,
\]
shows that $F$ is a Fredholm map of index equal to the index of the operator 
\[
\mR \times \mH \rightarrow \mH, \qquad 
(s,u) \rightarrow df(x)u,
\]
as $dF$ is a finite rank perturbation of the operator map above. Thus the index of $F$ is $\ind f+1$. Moreover, if $K\subset \mH$ is compact then the closed set $F^{-1}(K)$ is also compact, being contained in the compact set
\[
f^{-1}\bigl(K+[0,1](z-y) \bigr) \times [0,1].
\] 
This shows that $F$ is proper. The fact that the segment between $y$ and $z$ consists of regular values of $f$ implies that $y$ is a regular value of $F(t,\cdot)$ for each $t\in [0,1]$. In particular, $y$ is a regular value of $F$ and of $g:=F(1,\cdot) = f - z + y$. Therefore, $(F^{-1}(y),dF)$ is the required framed cobordism between $(f^{-1}(y),df)$ and $(g^{-1}(y),dg)= (f^{-1}(z),df)$.
\end{proof}

\begin{proposition} 
\label{prop:homotopycobordant}
Suppose that $f,g:M\rightarrow \mH$ are proper Fredholm maps which are proper Fredholm homotopic.  Let $y$ be a regular value of $f$ and $z$ a regular value of $g$. Then the Pontryagin manifold of $f$ at $y$ is framed cobordant to the Pontryagin manifold of $g$ at $z$.
\end{proposition}

\begin{proof} 
Since the set of regular values of proper Fredholm maps is open, we can find a convex open neighborhood $U$ of $y$ consisting of regular values of $f$. Similarly, let $V$ be a convex open neighborhood of $z$ consisting of regular values of $g$. Let $F:[0,1]\times M \rightarrow \mH$ be a proper Fredholm homotopy between $f$ and $g$. Consider the smooth map $G:[0,1]\times M\rightarrow \mH$ given by
\[
G(t,x):=F(t,x)-\omega(t)(z-y),
\] 
where $\omega:[0,1] \rightarrow [0,1]$ is a smooth function such that $\omega(t)=0$ for $t\in [0,1/3]$ and $\omega(t)=1$ for $t\in[2/3,1]$. Arguing as in the proof of Lemma \ref{lem:indeppert}, we see that $G$ is a proper Fredholm homotopy. Since the set of regular values of Fredholm maps is dense, we can find a point $y'\in U$ which is a regular value of $G$ and which is so close to $y$ that $z+y'-y$ belongs to $V$. By Lemma \ref{lem:indeppert}, $(f^{-1}(y),df)$ and $(f^{-1}(y'),df)$ are framed cobordant, and so are $(g^{-1}(z),dg)$ and $(g^{-1}(z+y'-y),dg)$. Since $G(0,\cdot) = F(0,\cdot) = f$ and $G(1,\cdot) = F(1,\cdot) - z + y = g - z + y$, $(G^{-1}(y'),dG)$ is a framed cobordism between $(f^{-1}(y'),df)$ and $(g^{-1}(z+y'-y),dg)$. The conclusion follows from the fact that framed cobordism is an equivalence relation. 
\end{proof}

\begin{remark}
As already observed, framed submanifolds of $M$ are defined in terms of a given trivialization $\{\tau_x : T_x M \rightarrow \mH\}_{x\in M}$ of the tangent bundle of $M$. Let  $\{\tilde{\tau}_x : T_x M \rightarrow \mH\}_{x\in M}$ be another trivialization of $TM$ and let 
\[
G : M \rightarrow GL(\mH), \qquad G(x):= \tau_x \circ \tilde{\tau}_x^{-1},
\]
be the corresponding transition map. Then the map $(X,A)\mapsto (X,AG)$ defines a bijection from the set of framed submanifolds of  $M$ with respect to the trivialization $\tau$ to the set of framed submanifolds of  $M$ with respect to the trivialization $\tilde{\tau}$. It is easy to see that this map preserves the equivalence relation given by framed cobordism and that it maps the Pontryagin manifold of $f$ at $y$ relative to the trivialization $\tau$ to the Pontryagin manifold of $f$ at $y$ relative to the trivialization $\tilde{\tau}$.
\end{remark}

\begin{remark}
The results in this section are also true if the codomain of $f$ is a general connected Hilbert manifold $N$ instead of $\mH$. 
\end{remark}

\begin{remark}
In \cite{sma65}, Smale introduced an invariant for proper homotopy classes of Fredholm maps of index $n\geq 0$: it is the unoriented bordism class of the inverse image of a regular value of the proper Fredholm map $f:M\rightarrow N$. The framed cobordism class of the Pontryagin manifold introduced here can be seen as a refinement of this invariant (in the spacial case $N=\mH$). The Smale invariant is quite poor when the topology of the domain $M$ is too simple. Indeed, assume that $n\geq 1$ and $\pi_k(M)=0$ for all $k\leq n$. Let $f: M \rightarrow N$ be a proper Fredholm map of index $n$ and let $y\in N$ be a regular value of $f$. 
By the arguments of Remark \ref{stablerem}, the tangent bundle of $f^{-1}(y)$ is stably trivial, and hence all the Stiefel-Whitney numbers of $f^{-1}(y)$ vanish. By a theorem of Thom (see for example \cite[Chapter VI]{sto68}), the Stiefel-Whitney numbers are a complete invariant for unoriented bordism and it follows that $f^{-1}(y)$ is nullbordant, so Smale's invariant vanishes in this case. Every unoriented bordism class does occur as Smale's invariant of some proper Fredholm map into $\mH$. To see this, let $X$ be a closed finite dimensional manifold and consider the projection $p:X\times \mH\rightarrow \mH$ to the second factor. The map $p$ is proper Fredholm of index $\dim X$ where every value is regular with inverse image $X$.
\end{remark}

\section{Every compact framed submanifold occurs as a Pontryagin manifold of some proper Fredholm map}

The results of this section use in a fundamental way the fact that we are working with $\mH$-valued maps. 

\begin{proposition}
\label{prop:submanifoldmap}
Let $n\in \mZ$ and let $(X,A)$ be an $n$-dimensional framed submanifold of $M$. Then there exists a proper Fredholm map $f:M\rightarrow \mH$ of index $n$ with $f^{-1}(0)=X$, $df|_X=A|_X$, and such that the maps $df:M \rightarrow \Phi_n(\mH)$ and $A: M \rightarrow \Phi_n(\mH)$ are homotopic relative $X$.
\end{proposition}

\begin{proof}
When $X$ is empty, for instance because $n<0$, this result reduces to the following statement: for every continuous map $A: M \rightarrow \Phi_n(\mH)$ there exists a proper Fredholm map $f: M \rightarrow \mH$ of index $n$ such that $f^{-1}(0)=\emptyset$ and $df$ is homotopic to $A$. This statement follows immediately from Theorem \ref{thm:totalextension} applied with $W=\emptyset$.

We now consider the case $X\neq \emptyset$. Denote by $NX$ the normal bundle of $X$ in $M$ and by $D$ the unit disc bundle $D=\{(x,v)\in NX\,|\,\norm x\leq 1\}$. Choose a tubular neighborhood $N$ of $X$ in $M$, which comes along with a diffeomorphism $\phi:\overline N\rightarrow D$ such that $\phi(x)=(x,0)$  and $d\phi(x)$ is the identity for every $x\in X$. The restriction of $A$ to $X$ induces the following isomorphism of vector bundles over $X$ 
\[
\widehat{A} : NX\rightarrow X\times \mH, \qquad \widehat{A}(x,v)=(x,A(x)v).
\]
We define $g:\overline{N}\rightarrow \mH$ as the composition
\[
 \overline{N}\xrightarrow{\phi} NX\xrightarrow{\widehat{A}}X\times \mH\xrightarrow{\pi_2} \mH.
\]
The projection onto the second factor $\pi_2$ is proper, because $X$ is compact, and since $\phi$ and $\widehat{A}$ are homeomorphisms onto their closed image we deduce that $g$ is proper. Moreover, $\pi_2$ is a Fredholm map of index $n=\dim X$ and so is $g$, because $\phi$ and $\widehat{A}$ are diffeomorphisms. By construction,
\[
g^{-1}(0) = X \qquad \mbox{and} \qquad dg(x) = A(x) \quad \forall x\in X.
\]
We claim that there is a continuous map $B: M \rightarrow \Phi_n(X)$ which coincides with $dg$ in a neighborhood of $X$ and is homotopic to $A$ relative $X$. Let us prove this claim. Since $A(x)=dg(x)$ for every $x\in X$, we can find an open neighborhood $W$ of $X$ such that $\overline{W}\subset N$ and for any $x\in W$ each convex combination between $A(x)$ and $dg(x)$ is Fredholm of index $n$. Let $\chi: M \rightarrow [0,1]$ be a smooth function which is supported in $W$ and takes the value 1 on a neighborhood of $X$. Then the map
\[
H :[0,1] \times M \rightarrow \Phi_n(\mH), \qquad 
H(t,x) := \left\{ \begin{array}{ll} t \chi(x) \, dg(x) + (1-t\chi(x)) A(x) & \mbox{for } x\in N, \\ A(x) & \mbox{for } x\in M \setminus N, \end{array} \right.
\]
is a homotopy relative $X$ between $A$ and a map $B:= H(1,\cdot)$ which coincides with $dg$ on a neighborhood of $X$. This concludes the proof of the claim.

Now fix a neighborhood $U$ of $X$ with $\overline U \subset W\subset N$, and notice that $0$ belongs to $g(U) \setminus g(\partial U)$, because $g^{-1}(0)=X$. By Theorem \ref{thm:totalextension} we can find a proper Fredholm map $f: M \rightarrow \mH$ such that $f|_U = g|_U$ and
\begin{enumerate}[(a)]
\item $df$ is homotopic to $B$ relative $U$;
\item there exists a neighborhood $Z\subset \mH$ of $0$ such that $f^{-1}(Z)=f^{-1}(Z)\cap U$. 
\end{enumerate}
By property (b) we have
\[
f^{-1}(0) = f^{-1}(0) \cap U = g^{-1}(0) \cap U = X.
\]
Moreover, $df|_X = dg|_X = A|_X$ and, by (a) and the properties of $B$, the maps $df$ and $A$ are homotopic relative $X$.
\end{proof}

\begin{remark}
\label{persu}
Let $(X,A)$ and $f:M \rightarrow \mH$ be as in Proposition \ref{prop:submanifoldmap}.
In particular, the Pontryagin manifold of $f$ at $0$ is framed cobordant to $(X,A)$ through the trivial framed cobordism $(W,B)$ which is defined by
\[
\begin{split}
W &:= [0,1]\times X = [0,1]\times f^{-1}(0), \\ B(t,x)(s,u) &:= H(t,x)u \quad \forall (t,x)\in [0,1]\times M, \; \forall (s,u)\in \mR \times \mH,
\end{split}
\]
where $H:[0,1]\times M \rightarrow \Phi_n(\mH)$ is a smooth homotopy relative $X$ between $df$ and $A$ which is independent of $t$ in a neighborhood of $t=0$ and in a neighborhood of $t=1$.
\end{remark}

Here is the cobordism version of the above result. The proof is similar. 

\begin{proposition}
\label{new}
Let $W\subset [0,1]\times M$ be an $(n+1)$-dimensional cobordism and let $B:[0,1]\times M \rightarrow \Phi_{n+1}(\mR\times \mH,\mH)$ be a framing of $W$. Then there exists a proper Fredholm map $h: [0,1]\times M \rightarrow \mH$ with $h^{-1}(0)=W$, $dh|_W=B|_W$, and such that the maps $dh$ and $B$ are homotopic relative $W$.
\end{proposition}

\section{Construction of proper Fredholm homotopies}

In order to show that proper Fredholm maps with cobordant Pontryagin manifolds are proper Fredholm homotopic, it is useful to consider the following particular case.

\begin{proposition}
\label{prop:equalframedsubmanifolds}
Let $f,g:M\rightarrow \mH$ be proper Fredholm maps of index $n\in \mZ$ having 0 as a regular value and such that 
\[
X:=f^{-1}(0)=g^{-1}(0), \qquad df|_X=dg|_X.
\] 
Assume moreover that the maps $df,dg:M\rightarrow \Phi_n(\mH)$ are homotopic relative $X$. Then $f$ and $g$ are proper Fredholm homotopic relative $X$.
\end{proposition}

\begin{proof} 
The proof consists of three steps. In the first step we show that we can reduce the general case to the situation in which $f$ and $g$ coincide in a neighborhood of $X$. In the second step we show that we can also assume the homotopy between $df$ and $dg$ to be relative a neighborhood of $X$. In the third and last step we apply the extension Theorem \ref{thm:totalextension} and obtain the result.  

\medskip

\noindent \emph{Step 1:} By the tubular neighborhood theorem, an open neighborhood of $X$ is diffeomorphic to the normal bundle of $X$. The fiber of this vector bundle can be identified with $\mH$, and since $GL(\mH)$ is contractible this bundle is trivial. Therefore, a neighborhood of $X$ in $M$ can be identified with the product $X\times \mH$. After this identification, the maps $f$ and $g$ can be seen as maps $F,G:X\times \mH\rightarrow \mH$ such that
\[
F(x,0)=G(x,0)=0 \quad \mbox{and} \quad dF(x,0)(u,v)=dG(x,0)(u,v)=v\quad \forall x\in X, \;u\in T_xX, \; v\in \mH.
\]
From Taylor's formula
\[
F(x,u) = u + o(\|u\|) \quad \mbox{and} \quad G(x,u) = u + o(\|u\|) \qquad \mbox{for } u\rightarrow 0, \mbox{ uniformly in } x\in X,
\]
we deduce that there is a positive number $\delta$ such that
\[
\langle F(x,u),u \rangle > 0 \quad \mbox{and} \quad \langle G(x,u),u \rangle >0 \qquad \forall (x,u)\in X\times \bigl(B_{\delta}(0)\setminus \{0\} \bigr),
\]
where $B_{\delta}(0)$ denotes the open ball of radius $\delta$ and center $0$ in $\mH$.
It follows that 
\begin{equation}
\label{conve}
(1-s) F(x,u) + s G(x,u) = 0 \mbox{ for some } (s,x,u)\in [0,1]\times X \times B_{\delta}(0)  \mbox{ if and only if } u =0.
\end{equation}
Since $dF(x,0)=dG(x,0)$, up to the choice of a smaller $\delta$ we may assume that for every $(x,u)\in X \times B_{\delta}(0)$ every convex combination between $dF(x,0)$ and $dG(x,0)$ is Fredholm. 

Let $\chi:[0,+\infty) \rightarrow [0,1]$ be a smooth function supported in $[0,\delta)$ and such that $\chi(t)=1$ for every $t\in [0,\delta/2]$. Consider the smooth homotopy
\[
H: [0,1]\times X \times \mH \rightarrow \mH, \qquad H(t,x,u) = \bigl(1-t\chi(\|u\|) \bigr)F(x,u)+t\chi(\|u\|)G(x,u).
\]
The map $H(0,\cdot)$ coincides with $F$, while $H(1,\cdot)$ agrees with $G$ on $X\times B_{\delta/2}(0)$. Moreover, $H(t,\cdot)$ coincides with $F$ on the complement of $X\times B_{\delta}(0)$, and the differential of this map at $(x,0)$ coincides with $dF(x,0) = dG(x,0)$, for every $x\in X$.
By property (\ref{conve}), the preimage of $0$ by $H$ is $[0,1]\times X \times \{0\}$. Moreover, the differential of $H(t,\cdot)$ at $(x,u)$  is a finite rank perturbation of the operator 
\[
 \bigl(1-t\chi(\|u\|) \bigr)dF(x,u)+t\chi(\|u\|)dG(x,u). 
 \]
This operator is Fredholm if $\|u\|< \delta$, by the above choice of $\delta$, and also for $\|u\|\geq \delta$, because in this case it coincides with the Fredholm operator $dF(x,u)$. Therefore, $H$ is a Fredholm map.

Since $H(t,\cdot)$ agrees with $F$ outside a bounded neighborhood of $X\times \{0\}$ in $X\times \mH$,  $H$ extends to a Fredholm homotopy 
\[
h: [0,1]\times M \rightarrow \mH,
\]
such that $h(t,\cdot)$ coincides with $f$ outside of a neighborhood of $X$ in $M$. By a time reparameterization, we can assume $h(t,\cdot)$ to be equal to $f$ for $t\in[0,1/2]$. By construction, we also have $h^{-1}(0)=[0,1] \times X$, $dh(t,\cdot) = df = dg$ on $X$, and $h(1,\cdot)=g$ on an open neighborhood of $X$. 

As Fredholm maps are locally proper and $[0,1]\times X$ is compact, there is an open neighborhood $V_0$ of $X$ in $M$ such that $h$ is proper on $[0,1]\times \overline{V}$, and hence also on the closure of the  open subset of $[0,1]\times M$
\[
V := \bigl([0,1]\times V_0\bigr) \cup \bigl([0,1/2)\times M\bigr),
\]
since $h(t,\cdot)=f$ for $t\in [0,1/2]$ and $f$ is proper. Let $U_0\subset M$ be an open neighborhood of $X$ such that $\overline{U_0} \subset V$ and set
\[
U:= \bigl([0,1]\times U_0\bigr) \cup \bigl([0,1/4)\times M\bigr),
\]
so that $[0,1]\times X\subset U$ and $\overline{U}\subset V$. Since $\partial U$ is disjoint from $[0,1]\times X$, the point $0$ does not belong to $h(\partial U)$. By applying Theorem \ref{thm:totalextension} to the Fredholm map $h$, to the sets $U$, $V$, $W:= [0,1]\times M$, and to the point $z=0$, we obtain a proper Fredholm map $\overline{h}:[0,1]\times M \rightarrow \mH$ which agrees with $h$ on $U$ and such that 
\[
\overline{h}^{-1}(0) = [0,1]\times X.
\]
Then the map $f_1:= \overline{h}(1,\cdot)$ has the following properties: $f_1$ is a proper Fredholm map,  $f_1^{-1}(0)=X$, $f_1$ is proper Fredholm homotopic to $f$ relative $X$, $df_1$ is homotopic to $df$ relative $X$, and $f_1 = g$ on a neighborhood of $X$. We have therefore reduced the general case to the situation in which $f$ and $g$ coincide in a neighborhood of their common Pontryagin manifold $X$.

\medskip

\noindent \emph{Step 2:} By Step 1, we can assume that $f$ and $g$ agree on an open neighborhood $N_0$ of $X$. Let $A:[0,1]\times M \rightarrow \Phi_n(\mH)$ be a homotopy relative $X$ between $df$ and $dg$. Since $A(t,\cdot)=df=dg$ on $X$, up to choice of a smaller neighborhood $N_0$ we may assume that for every $(t,x)\in [0,1]\times N_0$ each convex combination between $A(t,x)$ and $df(x)=dg(x)$ is Fredholm of index $n$. Now let $\psi: M \rightarrow [0,1]$ be a smooth function  with support in $N_0$ and such that $\psi=1$ on an open neighborhood $N_1$ of $X$. The map
\[
B(t,x) := (1-\psi(x)) A(t,x) + \psi(x) df(x) = (1-\psi(x)) A(t,x) + \psi(x) dg(x), \qquad (t,x)\in [0,1]\times M,
\]
is smooth, takes values into $\Phi_n(\mH)$ and satisfies
\[
\begin{split}
B(0,x) &= df(x) \mbox{ and } B(1,x) = dg(x) \quad \forall x\in M, \\ B(t,x) &= df(x) = dg(x) \quad \forall (t,x)\in [0,1]\times N_1.
\end{split}
\]
Therefore, $df$ and $dg$ are homotopic relative the open neighborhood $N_1$ of $X$.

\medskip

\noindent  \emph{Step 3:} Let $N_2$ and $N_3$ be open neighborhoods of $X$ with $\overline{N_3} \subset N_2 \subset \overline{N_2} \subset N_1$ and consider the following open subsets of $[0,1]\times M$
\[
\begin{split}
W &:= \Bigl( \bigl([0,1] \setminus [1/3,2/3] \bigr) \times M \Bigr) \cup \Bigl( [0,1] \times N_1 \Bigr), \\
V &:= \Bigl( \bigl([0,1] \setminus [1/4,3/4] \bigr) \times M \Bigr) \cup \Bigl( [0,1] \times N_2 \Bigr), \\
U &:= \Bigl( \bigl([0,1] \setminus [1/5,5/6] \bigr) \times M \Bigr) \cup \Bigl( [0,1] \times N_3 \Bigr),
\end{split}
\]
which satisfy $\overline{U} \subset V \subset \overline{V} \subset W$.
The map
\[
k: W \rightarrow \mH, \qquad k(t,x) = \left\{ \begin{array}{ll} f(x) & \mbox{for } t<1/3, \\ g(x) & \mbox{for } t>2/3, \\ f(x) = g(x) & \mbox{for } 1/3\leq t \leq 2/3, \end{array} \right.
\]
is smooth and locally constant in $t$. Moreover,
\[
df(t,x)[(s,u)] = \left\{ \begin{array}{ll} df(x)[u] & \mbox{for } t<1/3, \\ dg(x)[u] & \mbox{for } t>2/3, \\ df(x)[u] = dg(x)[u] & \mbox{for } 1/3\leq t \leq 2/3, \end{array} \right. \quad \forall (t,x) \in W, \; (s,u) \in \mR \times \mH,
\]
so $k$ is Fredholm of index $n+1$. From the properness of $f$ and $g$ we deduce that $k$ is proper on $\overline{V}$. Moreover, $k^{-1}(0) = [0,1]\times X$, and hence $0$ does not belong to the image of $\partial U$ by $k$. We can reparametrize the homotopy $B$ relative $N_1$ between $df$ and $dg$ so that $B(t,\cdot)=df$ for $t<1/3$ and $B(t,\cdot)=dg$ for $t>2/3$. Then the continuous map
\[
\begin{split}
K : [0,1]\times M &\rightarrow \Phi_{n+1}(\mR \times \mH, \mH), \\ K(t,x)[(s,u)] &= B(t,x)[u] \qquad \forall (t,x) \in [0,1]\times M, \; (s,u)\in \mR\times \mH,
\end{split}
\]
satisfies $K|_W=dk$. By applying Theorem \ref{thm:totalextension} to the data $k$, $K$, $U$, $V$, $W$ and $z=0$, we obtain a proper Fredholm map $\overline{k} : [0,1] \times M \rightarrow \mH$ of index $n+1$ such that $\overline{k}=k$ on $U$. This map is a proper Fredholm homotopy between $f$ and $g$ relative $X$.
\end{proof}

\section{The homotopy classification theorem for proper Fredholm maps}

We can finally determine the space $\mathcal{F}^{\mathrm{prop}}[M,\mH]$ of all proper Fredholm maps from $M$ into $\mH$ modulo proper Fredholm homotopies. By a particular case of Proposition \ref{prop:homotopycobordant}, the framed cobordism class of the Pontryagin manifold of a proper Fredholm map $f$ at a regular value $y$ does not depend on the regular value $y$. Therefore, when we consider framed submanifolds of $M$ modulo framed cobordism we can talk about the Pontryagin manifold of a map without specifying the regular value. The next result is Theorem \ref{main2} from the Introduction:

\begin{theorem}
\label{thm:properframed}
Let $f,g:M\rightarrow \mH$ be proper Fredholm maps. Then $f$ is proper Fredholm homotopic to $g$ if and only if the Pontryagin manifolds of $f$ and $g$ are framed cobordant. Moreover, the map which to each proper Fredholm map $f:M \rightarrow \mH$ of index $n$ associates the framed cobordism class of its Pontryagin manifold at some regular value induces a bijection from $\mathcal{F}_n^{\mathrm{prop}}[M,\mH]$ to the quotient space of framed $n$-dimensional compact submanifolds of $M$ modulo framed cobordism.
\end{theorem}

\begin{proof}
The ``only if'' part of the first statement is Proposition~\ref{prop:homotopycobordant}. Suppose that $f,g$ have framed cobordant Pontryagin manifolds, without loss of generality at a common regular value $y\in \mH$ for both $f$ and $g$.
Let $(W,B)$ be the corresponding framed cobordism between $(f^{-1}(y),df)$ and $(g^{-1}(y),dg)$. By Proposition~\ref{new} there exists a proper Fredholm map $h:[0,1]\times M\rightarrow \mH$ such that 
\[
h^{-1}(y) = W, \qquad dh|_W = B|_W
\]
and $dh$ is homotopic to $B$ relative $W$. By the properties of $(W,B)$, the proper Fredholm maps $h_0:= h(0,\cdot)$ and $h_1:= h(1,\cdot)$ have $y$ as a regular value and satisfy
\[
h_0^{-1}(y) = f^{-1}(y), \quad dh_0|_{f^{-1}(y)} = df|_{f^{-1}(y)}, \qquad h_1^{-1}(y) = g^{-1}(z), \quad dh_1|_{g^{-1}(z)} = dg|_{g^{-1}(z)}.
\]
From the fact that $dh$ and $B$  are homotopic relative $W$, we deduce that $dh_0$ is homotopic to $df$ relative $X$ and $dh_1$ is homotopic to $dg$ relative $X$. Therefore, Proposition~\ref{prop:equalframedsubmanifolds} implies that $h_0$ is proper Fredholm homotopic to $f$, and similarly $h_1$ is proper Fredholm homotopic to $g$. Since $h_0$ and $h_1$ are proper Fredholm homotopic, we conclude that $f$ and $g$ are proper Fredholm homotopic. This concludes the proof of the first statement.

This statement implies that the map which is defined in the second statement is well defined and injective. The surjectivity of this map is a consequence of Proposition \ref{prop:submanifoldmap} by what we observed in Remark \ref{persu}.
\end{proof}

\section{Computation of framed cobordisms for non-positive index}
\label{sec:nonpositive}
If $n$ is a negative integer, an $n$-dimensional framed submanifold of $M$ is just a pair $(\emptyset,A)$, where $A$ is a continuous map from $M$ into $\Phi_n(\mH)$. Two such $n$-dimensional framed submanifolds $(\emptyset,A_0)$ and $(\emptyset, A_1)$ are framed cobordant if and only if the maps $A_0$ and $A_1$ are homotopic. Therefore, Theorem \ref{thm:properframed} has the following direct consequence, which is Theorem \ref{main3} from the Introduction.

\begin{theorem}
Let $n$ be a negative integer. Then two proper Fredholm maps $f,g: M \rightarrow \mH$ of index $n$ are proper Fredholm homotopic if and only if their differentials are homotopic as maps from $M$ into $\Phi_n(\mH)$. Moreover, the map $f\mapsto df$ induces a bijection
\[
\mathcal{F}_n^{\mathrm{prop}}[M,\mH] \cong [M,\Phi_n(\mH)].
\]
\end{theorem}
If $M$ has to homotopy type of a compact Hausdorff space the Atiyah-J\"anich theorem states that $[M,\Phi_n(\mH)]$ is isomorphic to the reduced $K$-theory of real vector bundles of $M$, hence this set of homotopy classes is reasonably computable. Now we want to study the case of index zero. In order to simplify the discussion, we assume throughout the remaining part of this section that the manifold $M$ is connected. 

A $0$-dimensional framed submanifold of $M$ is a pair $(X,A)$ where $X$ is a finite subset of $M$ and $A:M \rightarrow \Phi_0(\mH)$ is a continuous map such that $A(x)\in GL(\mH)$ for every $x\in X$. Our aim is to understand when two such $0$-dimensional framed submanifolds $(X_0,A_0)$ and $(X_1,A_1)$ are framed cobordant. Obvious necessary conditions are that the sets $X_0$ and $X_1$ have the same parity and that the maps $A_0$ and $A_1$ are homotopic. These conditions are in general not sufficient, as we are going to show.

We recall that the fundamental group of $\Phi_0(\mH)$ is isomorphic to $\mZ_2$ and, since $GL(\mH)$ is contractible, the relative homotopy set $\pi_1(\Phi_0(\mH),GL(\mH))$ has two elements. We endow $\pi_1(\Phi_0(\mH),GL(\mH))$ with a group structure as follows: For two paths $\gamma,\gamma':[0,1]\rightarrow \Phi_0(\mH)$ with $\gamma(0),\gamma'(0),\gamma(1),\gamma'(1)\in GL(\mH)$ we choose a path $\beta:[0,1]\rightarrow GL(\mH)$ with $\beta(0)=\gamma(1)$, and $\beta(1)=\gamma'(0)$ and define the product $[\gamma][\gamma']$ to be the homotopy class relative endpoints of the concatenation $\gamma\#\beta\#\gamma'$. This does not depend on the choices made and the group $\pi_1(\Phi_0(\mH),GL(\mH))$ is isomorphic to $\mZ_2$. The identity element of this group is given by the homotopy classes of paths which are homotopic with fixed ends to a path fully contained in $GL(\mH)$.
The isomorphism
\[
\pi_1(\Phi_0(\mH),GL(\mH)) \cong \mZ_2
\]
is given by the mod 2 intersection number with the codimension 1 set $\Phi_0(\mH)\setminus GL(\mH)$. This intersection number can be defined in the following way: The set
\[
\Sigma:= \{A \in \Phi_0(\mH) \mid \dim \ker A = 1\}
\]
is a smooth hypersurface in $\Phi_0(\mH)$ whose closure is $\Phi_0(\mH) \setminus GL(\mH)$. Any continuous path $A: [0,1] \rightarrow \Phi_0(\mH)$ with end points in $GL(\mH)$ is homotopic with fixed ends to a smooth path $\tilde{A}$ which meets $\Phi_0(\mH) \setminus GL(\mH)$ only in $\Sigma$ and transversally. Then the mod 2 intersection number of $A$ with $\Phi_0(\mH) \setminus GL(\mH)$ is defined as the parity of the finite set $\tilde{A}^{-1}(\Sigma)$. A standard argument shows that this definition does not depend on the choice of the homotopic path $\tilde{A}$. See \cite{fpr92,fpr94} for more details.

If $A:[0,1] \rightarrow \Phi_0(\mH)$ is a continuous path with end points in $GL(\mH)$, the triviality of $A$ in $\pi_1(\Phi_0(\mH),GL(\mH))$ can also be checked by looking at the pull-back of the determinant bundle 
$\det \rightarrow \Phi(\mH)$ by $A$, see Remark~\ref{rem:det}. Indeed, $A^*(\det)$ is a line bundle over $[0,1]$, whose fiber at the end-points is the line $\Lambda^0(\{0\}) \otimes \Lambda^0(\{0\})^* \cong \mR \otimes \mR^*$, so it defines a line bundle over $[0,1]/\{0,1\}\cong S^1$, and $A$ is trivial in $\pi_1(\Phi_0(\mH),GL(\mH))$ if and only if the latter line bundle is trivial. See \cite{wan05,tw09} for more details.

It will be convenient to make use of the following definition:

\begin{definition}
A continuous map $A:M \rightarrow \Phi_0(\mH)$ is said to be orientable if the induced homomorphism
\[
A_* : \pi_1(M) \rightarrow \pi_1(\Phi_0(\mH)) \cong \mZ_2
\]
is zero. Otherwise it is said to be non-orientable.
\end{definition}

The orientability of $A:M \rightarrow \Phi_0(\mH)$ depends only on the homotopy class of $A$. Therefore, the set $[M,\Phi_0(\mH)]$ has a partition
\[
[M,\Phi_0(\mH)] = [M,\Phi_0(\mH)]_{\mathrm{or}} \sqcup [M,\Phi_0(\mH)]_{\mathrm{no}},
\]
where
\[
\begin{split}
[M,\Phi_0(\mH)]_{\mathrm{or}} :&= \{ [A] \in [M,\Phi_0(\mH)] | A \mbox{ orientable} \}, \\
[M,\Phi_0(\mH)]_{\mathrm{no}} :&= \{ [A] \in [M,\Phi_0(\mH)] | A \mbox{ non-orientable} \}.
\end{split}
\]
Continuous maps $A:M \rightarrow \Phi_0(\mH)$ with a simply connected domain $M$ are always orientable.

Let $(X,A)$ be a $0$-dimensional framed submanifold of $M$ such that $A$ is orientable. In this case, we can associate to $(X,A)$ a number $|\mathrm{deg}| (X,A)$ in
\[
\mN_0 := \{0\} \cup \mN
\]
which we call absolute degree of $(X,A)$ and which was introduced by Fitzpatrick, Pejsachowicz and Rabier in \cite{fpr92} and \cite{fpr92}. The definition of $|\mathrm{deg}| (X,A)$ goes as follows. Let $x,x'$ be points in $X$ and let $\gamma:[0,1] \rightarrow M$ be a continuous path joining them. We declare the points $x$ and $x'$ to be equivalent if the path
\[
A\circ \gamma : \bigl([0,1], \{0,1\}\bigr) \rightarrow \bigl( \Phi_0(\mH), GL(\mH) \bigr)
\]
defines the trivial element of $\pi_1(\Phi_0(\mH), GL(\mH))$. The fact that $A_* : \pi_1(M) \rightarrow \pi_1(\Phi_0(\mH)) \cong \mZ_2$ is the zero homomorphism implies that this notion does not depend on the choice of the path $\gamma$ joining $x$ and $x'$. Moreover, from the group structure of $\pi_1(\Phi_0(\mH),GL(\mH))$ we deduce that this is an equivalence relation. This equivalence relation has at most two equivalence classes: if $x$ is not equivalent to $x'$, $x'$ is not equivalent to $x''$ and $\gamma$ and $\gamma'$ are paths joining $x$ to $x'$ and $x'$ to $x''$, respectively, then $A\circ \gamma$ and $A\circ \gamma'$ are non-trivial in $\pi_1(\Phi_0(\mH), GL(\mH))\cong \mZ_2$ and hence their juxtaposition is trivial, implying that $x$ and $x''$ are equivalent. If $X=X^+ \sqcup X^-$ is the partition of $X$ induced by this equivalence relation, we define the non-negative integer $|\mathrm{deg}|(X,A)$ as
\[
|\mathrm{deg}|(X,A) := \bigl| |X^+| - |X^-| \bigr| \in \mN_0.
\]
Notice that the parity of this integer coincides with the parity of $|X|=|X^+| + |X^-|$.

\begin{theorem}
\label{indexzero}
The two $0$-dimensional framed submanifolds $(X_0,A_0)$ and $(X_1,A_1)$ of the connected manifold $M$ are framed cobordant if and only if the following conditions are satisfied:
\begin{enumerate}[(i)]
\item The maps $A_0,A_1: M \rightarrow \Phi_0(\mH)$ are homotopic.
\item  If $A_0$ and $A_1$ are orientable, then $|\mathrm{deg}|(X_0,A_0) = |\mathrm{deg}|(X_1,A_1)$. If $A_0$ and $A_1$ are non-orientable, then $|X_0|=|X_1|$ mod 2.
\end{enumerate}
Furthermore, for every $k\in \mN_0$ the following statements hold: 
\begin{enumerate}[(a)]
 \item For every continuous map $\tilde{A}:M \rightarrow \Phi_0(\mH)$ there exists a $0$-dimensional framed submanifold $(X,A)$ of $M$ such that $|X|=k$ and $A$ is homotopic to $\tilde{A}$.
\item For every orientable continuous map $\tilde{A}:M \rightarrow \Phi_0(\mH)$ there exists a $0$-dimensional framed submanifold $(X,A)$ of $M$ such that $|\mathrm{deg}|(X,A)=k$ and $A$ is homotopic to $\tilde{A}$.

  \end{enumerate}
\end{theorem}

The proof of the above result is given at the end of this section. Now let $f:M \rightarrow \mH$ be a proper Fredholm map of index zero. 

\begin{definition}
The index zero Fredholm map $f: M \rightarrow \mH$ is said to be orientable if $df$ is orientable. Otherwise it is said to be non-orientable.
\end{definition}

If $f$ is orientable, we can define
\[
|\mathrm{deg}|(f):= |\mathrm{deg}|(f^{-1}(y),df) \in \mN_0,
\]
where $y$ is a regular value of $f$. This definition does not depend on the choice of the regular value $y$: indeed, the Pontryagin manifolds of $f$ corresponding to two regular values are framed cobordant (see Proposition \ref{prop:homotopycobordant}) and the claim follows from the above theorem. The parity of $|\mathrm{deg}|(f)$ is nothing else but the Smale mod 2 degree of $f$:
\[
|\mathrm{deg}|(f) \mod 2 = \deg_2(f).
\]
Together with Theorem \ref{thm:properframed}, the above theorem has the following consequence, which is Theorem \ref{main4} from the Introduction.

\begin{corollary}
\label{indexzeromaps}
Two proper Fredholm maps $f_0,f_1: M \rightarrow \mH$ of index 0 on the connected manifold $M$ are proper Fredholm homotopic if and only if the following conditions are satisfied:
\begin{enumerate}[(i)]
\item The maps $df_0,df_1: M \rightarrow \Phi_0(\mH)$ are homotopic.
\item If $f_0$ and $f_1$ are orientable, then $|\mathrm{deg}|(f_0)=|\mathrm{deg}|(f_1)$. If $f_0$ and $f_1$ are non-orientable, then $\deg_2(f_0)=\deg_2(f_1)$.
\end{enumerate}
Moreover, the map
\[
f \mapsto \left\{ \begin{array}{ll} \bigl([df], |\mathrm{deg}|(f)) & \mbox{if $f$ is orientable} \\ \bigl([df], \deg_2(f)) & \mbox{if $f$ is  non-orientable} \end{array} \right.
\]
induces a bijection
\[
\mathcal{F}_0^{\mathrm{prop}}[M,\mH] \cong \Bigl( [M,\Phi_0(\mH)]_{\mathrm{or}} \times \mN_0 \Bigr)\sqcup \Bigl( [M,\Phi_0(\mH)]_{\mathrm{no}} \times \mZ_2\Bigr).
\]
\end{corollary}

If $M$ is simply connected, or more generally if $\pi_1(M)$ has no non-trivial homomorphisms into $\mZ_2$, the bijection of the above theorem reduces to
\[
\mathcal{F}_0^{\mathrm{prop}}[M,\mH] \cong  [M,\Phi_0(\mH)] \times \mN_0, \qquad [f] \mapsto \bigl([df],|\mathrm{deg}|(f) \bigr).
\]

The remaining part of this section is devoted to the proof of Theorem \ref{indexzero}, for which we need some preliminary Lemmas.

\begin{lemma}
\label{parity}
Let $\alpha:[0,1]\rightarrow \Phi_0(\mH)$ be a continuous path with $\alpha(0)$ and $\alpha(1)$ in $GL(\mH)$. Then the path
\[
\tilde{\alpha}: [0,1] \rightarrow \Phi_1(\mR\times \mH,\mH), \qquad \tilde{\alpha} = 0 \oplus \alpha,
\]
is homotopic with fixed ends to a path
\[
\beta: [0,1] \rightarrow \Phi_1(\mR\times \mH,\mH)
\]
such that $\dim \ker \beta(t) =1$ for every $t\in [0,1]$. Notice that 
\[
\ker \beta(0) = \ker \tilde{\alpha}(0) = \mR \times \{0\} = \ker \tilde{\alpha}(1) = \ker \beta(1),
\]
so $\ker \beta$ defines a line bundle over $[0,1]/\{0,1\} \cong S^1$. If $\tilde{\alpha}$ and $\beta$ are homotopic paths of the form above then the following facts are equivalent:
\begin{enumerate}[(i)]
\item $[\alpha]$ is trivial in $\pi_1(\Phi_0(\mH),GL(\mH))$;
\item the line bundle $\ker \beta \rightarrow [0,1]/\{0,1\} \cong S^1$ is trivial.
\end{enumerate}
\end{lemma}

\begin{proof}
We first notice that if $\beta_0$ and $\beta_1$ are paths in $\Phi_1(\mR \times \mH,\mH)$ which are homotopic with fixed ends and satisfy
\[
\dim \ker \beta_0(t) = \dim \ker \beta_1(t) = 1 \;\; \forall t\in [0,1], \quad \ker \beta_0(0)=\ker \beta_0(1) = \ker \beta_1(0)=\ker \beta_1(1),
\]
then the line bundle $\ker \beta_0 \rightarrow [0,1]/\{0,1\}$ is trivial if and only if the line bundle $\ker \beta_1 \rightarrow [0,1]/\{0,1\}$ is trivial. Indeed, by considering the pull back of the determinant bundle $\det \rightarrow \Phi_1(\mR \times \mH,\mH)$ by the homotopy between $\beta_0$ and $\beta_1$ we see that the line bundle $\beta_0^*(\det)\rightarrow [0,1]/\{0,1\}$ is trivial if and only if the line bundle $\beta_1^*(\det)\rightarrow [0,1]/\{0,1\}$ is trivial. Since $\beta_0(t)$ and $\beta_1(t)$ are surjective for every $t\in [0,1]$, these line bundles are canonically identified with the line bundles $\Lambda^1(\ker \beta_0)\rightarrow [0,1]/\{0,1\}$ and $\Lambda^1(\ker \beta_1)\rightarrow [0,1]/\{0,1\}$, and the conclusion follows from the fact that a line bundle $L$ is trivial if and only if the associated line bundle $\Lambda^1(L)$ is trivial. Thanks to this observation, in the second part of the statement we can assume that $\beta$ is the path which is constructed in the first part.

We treat first the case in which $\alpha$ is trivial in $\pi_1(\Phi_0(\mH),GL(\mH))$. In this case, $\alpha$ is homotopic with fixed ends to a path $\gamma:[0,1] \rightarrow \Phi_0(\mH)$ which takes values in $GL(\mH)$. Then $\tilde{\alpha}$ is homotopic with fixed ends to the path $\beta:= 0\oplus \gamma$, which satisfies $\ker \beta(t) = \mR \times \{0\}$ for every $t\in [0,1]$. In this case, both (i) and (ii) hold.

Now we consider the case in which $\alpha$ is non-trivial in $\pi_1(\Phi_0(\mH),GL(\mH))$. We claim that $\alpha$ is homotopic with fixed ends to a smooth path $\gamma$ in $\Phi_0(\mH)$ of the following form: Let $\mH = L \oplus L^{\perp}$ be an orthogonal splitting of $\mH$ with $\dim L=1$. Then $\gamma(0)=\alpha(0)$, $\gamma(1)=\alpha(1)$, $\gamma(t)$ is invertible for all $t\in [0,1/4] \cup [3/4,1]$, and 
\[
\gamma\left( \frac{1}{2} + \frac{t}{4} \right) = \left( \sin \Bigl( \frac{\pi}{2} t \Bigr)  I_L \right) \oplus  I_{L^{\perp}} \qquad \forall t\in [-1,1].
\]
The existence of such a path follows from the fact that $GL(\mH)$ is connected. In order to show that $\gamma$ is homotopic with fixed ends to $\alpha$, it is enough to check that also $\gamma$ is non-trivial in $\pi_1(\Phi_0(\mH),GL(\mH))$, and this follows from the fact that the intersection number of $\gamma$ with $\Phi_0(\mH)\setminus GL(\mH)$ is 1: indeed, the smooth path $\gamma$ meets $\Phi_0(\mH)\setminus GL(\mH)$ transversally at $\Sigma$, and $\gamma^{-1}(\Sigma)=\{1/2\}$.

It follows that $\tilde{\alpha}$ is homotopic with fixed ends to $\tilde{\gamma}:=0 \oplus \gamma$. The path $\tilde{\gamma}: [0,1] \rightarrow \Phi_1(\mR\times \mH,\mH)$ is homotopic with fixed ends to the path $\beta: [0,1] \rightarrow \Phi_1(\mR\times \mH,\mH)$ which is defined by the identities
\[
\begin{split}
\beta(t) &= \tilde{\gamma}(t) \qquad \forall t\in [0,1/4]\cup [3/4,1], \\ \beta \left( \frac{1}{2} + \frac{t}{4} \right) &= \left( \cos \Bigl( \frac{\pi}{2} t \Bigr)  I_{\mR} \right) \oplus \left( \sin \Bigl( \frac{\pi}{2} t \Bigr)  I_L \right) \oplus  I_{L^{\perp}} \qquad \forall t\in [-1,1].
\end{split}
\]
The kernel of $\beta(t)$ is $\mR\times \{0\}$ for $t\in [0,1/4]\cup [3/4,1]$, while for $t\in [1/4,3/4]$ this kernel describes a path of lines in the plane $\mR \times L \times \{0\}$ which starts and ends at $\mR\times \{0\} \times \{0\}$ and makes a rotation of angle $\pi$. Therefore, the line bundle $\ker \beta \rightarrow [0,1]/\{0,1\}$ is non-trivial. In this case, neither (i) nor (ii) holds.  
\end{proof}

Here is a first consequence of the above Lemma: 

\begin{lemma}
\label{LL1}
Let $(W,B)$ be a 1-dimensional framed cobordism starting at some $0$-dimensional framed submanifold $(X,A)$ of $M$. Let $x_0$ and $x_1$ be distinct points in $X$ such that $(0,x_0)$ and $(0,x_1)$ belong to the same connected component of $W$, which is parametrized by an embedding 
\[
w=(\tau,\gamma): [0,1] \rightarrow [0,1]\times M
\]
such that $\gamma(0)=x_0$ and $\gamma(1)=x_1$. Then the path $A\circ \gamma$ is non-trivial in $\pi_1(\Phi_0(\mH),GL(\mH))$.
\end{lemma}

\begin{proof}
Notice that the first component $\tau$ of the embedding $w$ must satisfy
\[
\tau(0)=\tau(1)=0 \qquad \mbox{and} \qquad \tau'(0)>0, \; \tau'(1)<0.
\]
The path 
\[
\beta: [0,1] \rightarrow \Phi_1(\mR\times \mH,\mH), \qquad \beta:= B\circ w,
\]
is such that $\dim \ker \beta(t)=1$ for every $t\in [0,1]$ and
\[
\beta(0) = 0 \oplus A(x_0), \qquad \beta(1) = 0 \oplus A(x_1).
\]
Let $\alpha:= A \circ \gamma : [0,1] \rightarrow \Phi_0(\mH)$ and set $\tilde{\alpha}:= 0 \oplus \alpha : [0,1] \rightarrow \Phi_1(\mR \times \mH,\mH)$. The paths $\tilde{\alpha}$ and $\beta$ are homotopic with fixed ends through the homotopy
\[
[0,1]\times [0,1] \rightarrow \Phi_1(\mR \times \mH,\mH), \qquad (s,t) \mapsto B\bigl(s\tau(t),\gamma(t)\bigr).
\]
The path $\dot{w}:[0,1] \rightarrow \mR \times \mH$ defines a non-vanishing section of the line bundle $\ker \beta \rightarrow [0,1]$ such that
\[
\begin{split}
\dot{w}(0) &= \big( \tau'(0),0) \quad \mbox{with } \tau'(0)>0, \\
\dot{w}(1) &= \big( \tau'(1),0) \quad \mbox{with } \tau'(1)<0.
\end{split}
\]
The fact that a non-vanishing section of $\ker \beta \rightarrow [0,1]$ takes values into two different components of 
\[
\ker \beta(0)\setminus \{0\} = (\mR\setminus \{0\}) \times \{0\} = \ker \beta(1)\setminus \{0\}
\]
at $t=0$ and $t=1$ implies that the line bundle $\ker \beta \rightarrow [0,1]/\{0,1\} \cong S^1$ is non-trivial. Then the last part of Lemma \ref{parity} implies that $[\alpha]=[A\circ \gamma]$ is non-trivial in $\pi_1(\Phi_1(\mH),GL(\mH))$. 
\end{proof}

This Lemma has the following converse:

\begin{lemma}
\label{LL2}
Let $(X,A)$ be a $0$-dimensional framed submanifold of $M$, let $x_0$ and $x_1$ be distinct points in $X$ and let $\gamma:[0,1] \rightarrow M$ be a continuous path from $x_0$ to $x_1$ such that $A\circ \gamma$ is non-trivial in $\pi_1(\Phi_0(\mH),GL(\mH))$. Then there is a $1$-dimensional framed cobordism $(W,B)$ from $(X,A)$ to a $0$-dimensional framed submanifold $(X',A')$ with $X'=X \setminus \{x_0,x_1\}$. Moreover, for every continuous path $\delta:[0,1] \rightarrow M$ connecting two points of $X'$, the path $A'\circ \delta$ is trivial in $\pi_1 (\Phi_0(\mH),GL(\mH))$ if and only if the path $A\circ \delta$ is.
\end{lemma}

\begin{proof}
By replacing it with a fixed ends homotopic path, we can assume that the path $\gamma:[0,1] \rightarrow M$ is smooth and has the following properties: $\gamma(t)=x_0$ for $t\in [0,1/3]$, $\gamma(t)=x_1$ for $t\in [2/3,1]$ and $\gamma|_{(1/3,2/3)}$ is an embedding into $M \setminus X$. Let $\tau: [0,1] \rightarrow [0,1/2]$ be a smooth function such that $\tau(t)=t$ for $t\in [0,1/3]$ and $\tau(t)=1-t$ for $t\in [2/3,1]$ 
and $\tau(t)>0$ for $t\in(0,1)$. Then the path
\[
w:[0,1] \rightarrow [0,1]\times M, \qquad w:=(\tau,\gamma),
\]
is an embedding, and the 1-dimensional manifold
\[
W:= w([0,1]) \cup \bigcup_{x\in X'} [0,1]\times \{x\}
\]
is a cobordism between $X$ and $X'=X\setminus \{x_0,x_1\}$. 

Let $Z\subset [0,1]\times M$ be the image of the smooth map
\[
z: [0,1] \times [0,1] \rightarrow [0,1]\times M, \qquad (s,t) \mapsto \bigl( s \tau(t), \gamma(t) \bigr),
\]
which maps $[0,1]\times \{0\}$ into $(0,x_0)$ and $[0,1]\times \{1\}$ into $(0,x_1)$. The map $z$ induces a homeomorphism from the quotient space, which is obtained from $[0,1]\times [0,1]$ by collapsing $[0,1]\times \{0\}$ and $[0,1]\times \{1\}$, onto $Z$, which is therefore a topological disc.

Set $\alpha:= A \circ \gamma$ and $\tilde{\alpha}:= 0 \oplus \alpha:[0,1] \rightarrow \Phi_1(\mR \times \mH,\mH)$. By the first part of Lemma \ref{parity}, $\tilde{\alpha}$ is homotopic with fixed ends to a path
\[
\widehat{\beta}:[0,1] \rightarrow  \Phi_1(\mR \times \mH,\mH)
\]
such that $\dim \ker \widehat{\beta}(t)=1$ for every $t\in [0,1]$. Due to the properties of $\tilde{\alpha}$, we may assume that
\[
\widehat{\beta}(t)= \tilde{\alpha}(t) \qquad \forall t\in [0,1/3]\cup [2/3,1].
\]
Since $[\alpha]=[A\circ \gamma]$ is non-trivial in  $\pi_1(\Phi_1(\mH),GL(\mH))$, the second part of Lemma \ref{parity} implies that the line bundle $\ker \widehat{\beta} \rightarrow [0,1]/\{0,1\}\cong S^1$ is non-trivial. Also the path
\[
t\mapsto \dot{w}(t) \mR 
\]
into the Grassmannian of lines in $\mR \times \mH$ defines a non-trivial line bundle over $[0,1]/\{0,1\}\cong S^1$, which coincides with the previous one on $[0,1/3]$ and on $[2/3,1]$. Then we can find a smooth path
\[
\Psi: [0,1] \rightarrow GL(\mR\times \mH)
\]
such that $\Psi(t)=I$ for $t\in [0,1/3]\cup [2/3,1]$ and $\Psi(t)\dot{w}(t)\in \ker \widehat{\beta}(t)$ for every $t\in [0,1]$. Since $GL(\mR\times \mH)$ is contractible, the pointwise multiplication
\[
\beta:[0,1] \rightarrow \Phi_1(\mR \times \mH,\mH), \qquad \beta(t):= \widehat{\beta}(t) \Psi(t),
\]
is still homotopic with fixed ends to $\tilde{\alpha}$. Let
\[
\eta: [0,1]\times [0,1] \rightarrow \Phi_1(\mR\times \mH,\mH), \qquad \eta(0,t) = \tilde{\alpha}(t), \quad \eta(1,t) = \beta(t),
\]
be such a homotopy, which we may assume to satisfy
\[
\eta(s,t) = \tilde{\alpha}(t) \qquad \forall (s,t) \in [0,1] \times \bigl( [0,1/3] \cup [2/3,1]\bigr).
\]
By construction,
\[
\ker \beta(t) = \mR \dot{w}(t) = T_{w(t)} W
\]
for every $t\in [0,1]$. Using the fact that 
\[
\bigl( \{0\} \times M \bigr) \cup Z \cup \bigl( [0,1] \times X' \bigr)
\]
is a retract of $[0,1]\times M$ and the properties of the map $z$, we can find a continuous map
\[
B: [0,1] \times M \rightarrow \Phi_1(\mR\times \mH,\mH)
\]
such that
\[
\begin{split}
B(s,x) &= 0 \oplus A(x) \qquad \forall (s,x)\in [0,1/4] \times M, \\
B(z(s,t)) &= \eta(s,t) \qquad \forall (s,t)\in [0,1]\times [0,1], \\
B(s,x) &= 0 \oplus A(x) \qquad \forall (s,x)\in [0,1] \times X'.
\end{split}
\]
Since the map
\[
\Phi_0(\mH) \rightarrow \Phi_1(\mR\times \mH,\mH), \qquad A \mapsto 0\oplus A,
\]
is a homotopy equivalence, by modifying $B$ in the open set $(3/4,1]\times M$, whose intersection with $W$ is 
\[
(3/4,1] \times X',
\]
we may assume that
\[
B(s,x) = 0 \oplus A'(x) \qquad \forall (s,x) \in [4/5,1]\times M,
\]
where $A'(x)=A(x)$ for all $x\in X'$. From the identities
\[
\begin{split}
\ker B(w(t)) &= \ker B(z(1,t)) = \ker \eta(1,t) = \ker \beta(t) = T_{w(t)} W \qquad \forall t\in [0,1], \\
\ker B(t,x) &= \ker 0 \oplus A(x) = \mR \times \{0\} = T_{(t,x)} W \qquad \forall (t,x) \in [0,1] \times X',
\end{split}
\]
we then deduce that $(W,B)$ is the required framed cobordism from $(X,A)$ to $(X',A')$.

In order to prove the last statement, we consider a continuous path $\delta:[0,1] \rightarrow M$ joining two points of $X'$. Since $B(s,x) = 0 \oplus A(x)$ for all $(s,x)\in [0,1] \times X'$, the map
\[
H: [0,1] \times [0,1] \rightarrow \Phi_0(\mH), \qquad H(s,t) = B(s,\delta(t))|_{\{0\} \times \mH},
\]
is a homotopy with fixed ends between $A\circ \delta$ and $A'\circ \delta$. The existence of this homotopy implies that $A'\circ \delta$ is trivial in $\pi_1(\Phi_0(\mH),GL(\mH))$ if and only if $A\circ \delta$ is. 
\end{proof}

Another consequence of Lemma \ref{parity} is the following:

\begin{lemma}
\label{LL3}
Let $(W,B)$ be a 1-dimensional framed cobordism between the 0-dimensional framed submanifolds $(X_0,A_0)$ and $(X_1,A_1)$. Let $x_0\in X_0$ and $x_1\in X_1$ be such that $(0,x_0)$ and $(1,x_1)$ belong to the same connected component of $W$, which is parametrized by the embedding
\[
w: [0,1] \rightarrow [0,1]\times M, \qquad w(0)=(0,x_0), \; w(1)=(1,x_1).
\]
Consider the homotopy 
\[
A : [0,1] \times M \rightarrow \Phi_0(\mH), \qquad A(t,x):= B(t,x)|_{\{0\} \times \mH} \quad \forall (t,x)\in [0,1]\times M.
\]
Then the path $A\circ w$ is trivial in $\pi_1(\Phi_0(\mH),GL(\mH))$.  
\end{lemma}

\begin{proof}
Due to the form of $B$, $B\circ w$ is homotopic with fixed ends to $0 \oplus A\circ w$ within $\Phi_1(\mR\times \mH,\mH)$. The kernel of $B\circ w(t)$ has dimension 1 for every $t\in [0,1]$ and is spanned by $\dot{w}(t)$. Since $\dot{w}(0)$ and $\dot{w}(1)$ belong to the same component of 
\[
\ker B(w(0)) \setminus \{0\} = \bigl(\mR \setminus \{0\} \bigr) \times \{0\} =  \ker B(w(1)) \setminus \{0\},
\]
the line bundle $\ker B\circ w \rightarrow [0,1]/\{0,1\}$ is trivial. Then the second part of Lemma \ref{parity} implies that $A\circ w$ is trivial in $\pi_1(\Phi_0(\mH),GL(\mH))$.
\end{proof}

And here is a converse statement to the above lemma:

\begin{lemma}
\label{LL4}
Let $A_0,A_1: M \rightarrow \Phi_0(\mH)$ be framings of the same finite set $X\subset M$. Let $A:[0,1] \times M \rightarrow \Phi_0(\mH)$ be a homotopy from $A_0$ to $A_1$ such that for every $x\in X$ the path $t\mapsto A(t,x)$ is trivial in $\pi_1(\Phi_0(\mH),GL(\mH))$. Then $(X,A_0)$ is framed cobordant to $(X,A_1)$.
\end{lemma}

\begin{proof}
By a reparametrization, we may assume that $A(t,\cdot)=A_0$ for every $t\in [0,1/3]$ and $A(t,\cdot)=A_1$ for every $t\in [2/3,1]$. Let $r>0$ be so small that the closed balls $\overline{B}_r(x)$ for $x\in X$ are pairwise disjoint. Since for every $x\in X$ the path 
\[
[1/3,2/3] \rightarrow \Phi_0(\mH), \qquad t \mapsto A(t,x)
\]
is homotopic with fixed ends to a path which is fully contained in $GL(\mH)$, we can modify the homotopy $A$ inside $[1/3,2/3]\times \bigcup_{x\in X} \overline{B}_r(x)$ so that $A(t,x)\in GL(\mH)$ for every $x\in X$ and every $t\in [0,1]$. After this modification, the pair $(W,B)$ with $W:= [0,1] \times X$ and $B:= 0 \oplus A :[0,1] \times M \rightarrow \Phi_1(\mR \times \mH,\mH)$ is a framed cobordism from $(X,A_0)$ to $(X,A_1)$.
\end{proof}

We can finally prove Theorem \ref{indexzero}

\begin{proof}[Proof of Theorem \ref{indexzero}] Let $(W,B)$ be a framed cobordism between the 0-dimensional framed manifolds $(X_0,A_0)$ and $(X_1,A_1)$. As we have already observed, $B$ induces a homotopy 
\[
A:[0,1] \times M \rightarrow \Phi_0(\mH), \qquad A(t,x) := B(t,x)|_{\{0\} \times \mH} \qquad \forall (t,x)\in [0,1] \times M,
\]
between $A_0$ and $A_1$, so (i) holds. We must prove that also (ii) holds.

If $A_0$ and $A_1$ are non-orientable, then (ii) just says that $|X_0|=|X_1|$ mod 2, and this is surely true because the finite sets $X_0$ and $X_1$ are cobordant. Therefore, we can assume that $A_0$ and $A_1$ are orientable. In this case, $A_0$ and $A_1$ induce equivalence relations on $X_0$ and $X_1$ which determine the non-negative integers $|\mathrm{deg}|(X_0,A_0)$ and $|\mathrm{deg}|(X_1,A_1)|$, and we must show that these integers coincide.

Let $W'$ be the subset of $W$ consisting of those components having one boundary element in $\{0\}\times M$ and the other in $\{1\}\times M$. The set $W'$ intersects $\{0\} \times M$ in the subset $X_0'\subset X_0$ and in the equipotent subset $X_1'\subset X_1$. The set $X_0\setminus X_0'$ consists of a disjoint union of pairs $x_0,x_0'$ such that $(0,x_0)$ and $(0,x_0')$ belong to the same component of $W$. The elements of such a pair are not equivalent under the equivalence relation induced by $A_0$, thanks to Lemma \ref{LL1}, and hence
\[
|\mathrm{deg}|(X_0,A_0) = |\mathrm{deg}|(X_0',A_0).
\]
It is easy to see that $(X_0,A_0)$ and $(X_0',A_0)$ are framed cobordant. Similarly, 
\[
|\mathrm{deg}|(X_1,A_1) = |\mathrm{deg}|(X_1',A_1)
\]
and $(X_1,A_1)$ is framed cobordant to $(X_1',A_1)$. By replacing $X_0$ by $X_0'$ and $X_1$ by $X_1'$, we can therefore assume that all components of $W$ are arcs with one end-point in $\{0\}\times X_0$ and the other in $\{1\}\times X_1$. Under this assumption, $W$ induces a bijection
\[
\varphi: X_0 \rightarrow X_1
\]
where $\varphi(x_0)$ is the point of $X_1$ such that $(0,x_0)$ and $(1,\varphi(x_0))$ are end-points of a component of $W$. In order to prove that $|\mathrm{deg}|(X_0,A_0) = |\mathrm{deg}|(X_1,A_1)$, it is enough to show that the bijection $\varphi$ preserves the two equivalence relations. In other words, we must prove the following fact: if $x_0,x_0'$ are distinct points in $X_0$ then $x_0$ is equivalent to $x_0'$ with respect to $A_0$ if and only if $x_1:=\varphi(x_0)$ is equivalent to $x_1':=\varphi(x_0')$ with respect to $A_1$. The connected component of $W$ having boundary $(0,x_0)$ and $(1,x_1)$ is parametrized by an embedding $w:[0,1] \rightarrow [0,1]\times M$ such that $w(0)=(0,x_0)$ and $w(1)=(1,x_1)$. Similarly, the connected component of $W$ having boundary $(0,x_0')$ and $(1,x_1')$ is parametrized by an embedding $w':[0,1] \rightarrow [0,1]\times M$ such that $w'(0)=(0,x_0')$ and $w'(1)=(1,x_1')$. Let $\alpha_0: [0,1] \rightarrow M$ and $\alpha_1:[0,1] \rightarrow M$ be continuous paths such that $\alpha_0(0)=x_0$, $\alpha_0(1)=x_0'$, $\alpha_1(0)=x_1$ and $\alpha_1(1)=x_1'$.  The homotopy $A$ induces a homotopy between the juxtapositions
\[
(A_0\circ \alpha_0) \# (A\circ w')
\]
and
\[
(A\circ w) \# (A_1\circ \alpha_1).
\]
By Lemma \ref{LL3}, the paths $A\circ w$ and $A \circ w'$ are trivial in $\pi_1(\Phi_0(\mH),GL(\mH))$. We conclude that the paths $A_0\circ \alpha_0$ and $A_1\circ \alpha_1$ are either both trivial or both non-trivial in $\pi_1(\Phi_0(\mH),GL(\mH))$, and hence $x_0$ is equivalent to $x_0'$ with respect to $A_0$ if and only if $x_1$ is equivalent to $x_1'$ with respect to $A_1$. This proves that (ii) holds.

\medskip

Now we assume that the framed $0$-dimensional manifolds $(X_0,A_0)$ and $(X_1,A_1)$ satisfy (i) and (ii). We wish to prove that they are framed cobordant. 

Since $A_0$ and $A_1$ are homotopic by (i), they are either both orientable or both non-orientable. 
We first consider the case in which they are orientable.
If $X_0$ contains two points $x_0$ and $x_0'$ which are not equivalent for the equivalence relation induced by $A_0$, then Lemma \ref{LL2} implies that $(X_0,A_0)$ is framed cobordant to $(X_0\setminus \{x_0,x_0'\},A_0')$ for some framing $A_0'$ which induces the same equivalence relation as $A_0$ on $X_0\setminus \{x_0,x_0'\}$. By the latter fact, removing $x_0$ and $x_0'$ does not change the value of $|\mathrm{deg}|$:
\[
|\mathrm{deg}| \bigl(X \setminus \{x_0,x_0'\}, A_0'\bigr) = |\mathrm{deg}|(X_0,A_0).
\]
By successively removing pairs of non-equivalent points from $X_0$, we can therefore assume that all the points of $X_0$ are equivalent with respect to the equivalence relation induced by $A_0$. Similarly, we can assume that all the points of $X_1$ are equivalent with respect to the equivalence relation induced by $A_1$. In this case, we have $|\mathrm{deg}|(X_0,A_0) = |X_0|$ and $|\mathrm{deg}|(X_1,A_1)=|X_1|$, so assumption (ii) guarantees that $X_0$ and $X_1$ have the same number of elements. By considering an isotopy 
\[
\phi: [0,1] \times M \rightarrow M
\]
such that $\phi(0,\cdot) = \mathrm{id}$ and $\phi(1,X_0)=X_1$, we can assume that $X_0$ and $X_1$ are indeed the same set, which we denote by $X$. Let $A:[0,1] \times M \rightarrow \Phi_0(\mH)$ be a homotopy between $A_0$ and $A_1$, whose existence is guaranteed by assumption (i). If $x$ and $x'$ are distinct elements of $X$ and $\gamma:[0,1]\rightarrow M$ is a path such that $\gamma(0)=x$ and $\gamma(1)=x'$, then the juxtapositions 
\[
A(\cdot,x) \# A_1\circ \gamma
\]
and
\[
A_0 \circ \gamma \# A(\cdot,x')
\]
are homotopic. Since the paths $A_1\circ \gamma$ and $A_0\circ \gamma$ are trivial in $\pi_1(\Phi_0(\mH),GL(\mH))$, the path $A(\cdot,x)$ is trivial if and only if the path $A(\cdot,x')$ is trivial. If $A(\cdot,x)$ is trivial for one $x\in X$, and hence for all $x\in X$, then Lemma \ref{LL4} implies that $(X,A_0)$ and $(X,A_1)$ are framed cobordant, as we wished to prove. If $A(\cdot,x)$ is non-trivial for one $x\in X$, and hence for all $x\in X$, then we can modify the homotopy $A$ as follows. We fix a path $T:[0,1] \rightarrow \Phi_0(\mH)$ such that $T(0)=T(1)=I$ and $T$ is non-trivial in $\pi_1(\Phi_0(\mH))$ and we define
\[
\tilde{A}(t,x) := T(t) A(t,x) \qquad \forall (t,x)\in [0,1]\times M.
\]
This map is a homotopy between $A_0$ and $A_1$ and $\tilde{A}(\cdot,x)$ is trivial in $\pi_1(\Phi_0(\mH),GL(\mH))$ for all $x\in X$, so the previous argument applies and $(X,A_0)$ and $(X,A_1)$ are framed cobordant. This concludes the case in which $A_0$ and $A_1$ are orientable.

Now assume that $A_0$ and $A_1$ are non-orientable. In this case, any pair of distinct points $x_0,x_1$ in $X_0$ can be connected by a path $\gamma:[0,1] \rightarrow M$ such that $A\circ \gamma$ is non-trivial in $\pi_1(\Phi_0(\mH),GL(\mH))$. Indeed, if some path $\alpha$ from $x_0$ to $x_1$ is such that $A\circ \alpha$ is trivial, then by juxtaposing $\alpha$ with a loop $\beta$ based at $x_1$ such that $A\circ \beta$ is non-trivial - which exists because $A$ is non-orientable, we obtain the required path. Thanks to Lemma \ref{LL1}, $(X_0,A_0)$ is framed cobordant to a $0$-dimensional framed submanifold $(X_0',A_0')$ such that $X_0$ is either empty or consists of one point and $|X_0|=|X_0'|$ mod 2. The same is true for $(X_1,A_1)$, and hence we may assume that both $X_0$ and $X_1$ have at most 1 element. Since $|X_0|=|X_1|$ mod 2 by (ii), $X_0$ and $X_1$ are either both empty or both singletons. 

Let $A:[0,1]\times M \rightarrow \Phi_0(\mH)$ be a homotopy from $A_0$ to $A_1$, which we may assume to be independent of time for $t$ close to 0 and for $t$ close to 1. If $X_0=X_1=\emptyset$, then the pair $(W,B)$ with $W=\emptyset$ and 
\[
B: [0,1] \times M \rightarrow \Phi_1(\mR \times \mH, \mH), \qquad B = 0 \oplus A,
\]
is a framed cobordism from $(X_0,A_0)$ to $(X_1,A_1)$. If $X_0$ and $X_1$ are singletons, up to applying an isotopy we may assume that $X_0=X_1=\{x\}$. We can also assume that the path $t\mapsto A(t,x)$ is non-trivial in $\pi_1(\Phi_0(\mH),GL(\mH))$, because otherwise we can replace the homotopy $A$ by the homotopy $T A$, where $T$ is a non-trivial loop based at $I$ as above. Then Lemma \ref{LL4} implies that $(X_0,A_0)$ and $(X_1,A_1)$ are framed cobordant.

\medskip

It remains to prove (a) and (b). Let $k$ be a non-negative integer and let $\tilde{A}:M \rightarrow \Phi_0(\mH)$ be a continuous map. 
Since $GL(\mH)$ is dense in $\Phi_0(\mH)$, we can find a continuous map $A: M \rightarrow \Phi_0(\mH)$ which is homotopic to $\tilde{A}$ and such that $A^{-1}(GL(\mH))$ is not empty. Let $X$ be a subset of cardinality $k$ of some connected component of the open set $A^{-1}(GL(\mH))$. Then $(X,A)$ is a $0$-dimensional framed submanifold of $M$ with $|X|=k$ and $A$ homotopic to $\tilde{A}$, so (a) holds. If $\tilde{A}$ is orientable, then the fact that $X$ is contained in a connected component of $A^{-1}(GL(\mH))$ implies that the points of $X$ are all pairwise equivalent with respect to the equivalence relation induced by $A$, and hence $|\mathrm{deg}|(X,A)=|X|=k$. This proves (b).
\end{proof}

\appendix 

\section{Extension of Fredholm maps}
\label{app:extension}

The aim of this appendix is to prove Theorem \ref{thm:extensionfredholm} on the extension of Fredholm maps. 

By the already cited theorem of Eells and Elworthy \cite{ee70}, every Hilbert manifold has an open embedding into its model. Therefore, we may assume that $M$ and $N$ are open subsets of $\mH$. The trivializations of $TM$ and $TN$ which we use to write the differential of a Fredholm map between $M$ and $N$ as a $\Phi(\mH)$-valued map on $M$ are the ones induced by these embeddings.

Up to the regularization of $A$, see Lemma \ref{smoothing}, and up to the choice of a smaller $V$, we may assume that the map $A$ is smooth. For any $x\in M$ set
\[
\delta_0(x) := \mathrm{dist} \bigl(g(x),N^c\bigr)>0, \quad \delta_1(x):= \mathrm{dist} \bigl( A(x), \Phi_n(\mathbb{H})^c \bigr)>0,
\]
where $N^c$ and $\Phi_n(\mathbb{H})^c$ denote the complements of $N$ and $\Phi_n(\mathbb{H})$ in $\mathbb{H}$ and $L(\mathbb{H})$, respectively. We set $\delta_0(x) :=+\infty$ if $N=\mathbb{H}$. Now we choose $r(x)>0$ to be so small that:
\renewcommand{\theenumi}{\alph{enumi}}
\renewcommand{\labelenumi}{(\theenumi)}
\begin{enumerate}
\item $B_{r(x)}(x)\subset M$;   
\item if $B_{r(x)}(x)\cap U\neq \emptyset$ then $B_{r(x)}(x)\subset V$;
\item $\mathrm{\overline{conv}}\, g(B_{r(x)}(x)) \subset \{z\in N \mid \|z-g(x)\|< \delta_0(x)/2\}$;
\item $\mathrm{\overline{conv}}\, A(B_{r(x)}(x)) \subset \{T\in \Phi_n(\mathbb{H}) \, \mid \, \|T-A(x)\|_{L(\mathbb{H})}< \delta_1(x)/2\}$;
\item $r(x) \sup \{\|A(y)\|_{L(\mathbb{H})} \mid y\in B_{r(x)}(x)\} < \delta_0(x)/4$;
\item $r(x) \sup \{\|dA(y)\|_{L(\mathbb{H},L(\mathbb{H}))}  \mid y\in B_{r(x)}(x)\} < \delta_1(x)/2$.
\end{enumerate}
Here $\mathrm{\overline{conv}}$ denotes the closed convex hull.

Let $\mathscr{U} = \{U_j \mid j\in J\}$ be a locally finite star-refinement of the open covering $\{B_{r(x)}(x) \mid x\in M\}$. That is,
\begin{equation}
\label{uno}
U_{j_1} \cap \dots \cap U_{j_k} \neq \emptyset \quad \implies \quad U_{j_1} \cup \dots \cup U_{j_k} \subset B_{r(x)} (x) \mbox{ for some } x\in M.
\end{equation}
Let $\{\varphi_j \mid j\in J\}$ be a smooth partition of unity subordinated to $\mathscr{U}$, and for every $j\in J$ fix a point $x_j$ in $U_j$. 
We define the smooth maps
\begin{eqnarray*}
\overline{f}(x) & := & \sum_{j\in J} \varphi_j(x) \Bigl( g(x_j) + \int_0^1 A\bigl(x_j + s(x-x_j) \bigr) (x-x_j) \, ds \Bigr), \quad \forall x\in M, \\
h(t,x) & := & t \overline{f}(x) + (1-t) g(x), \quad \forall (t,x)\in [0,1]\times M.
\end{eqnarray*}

Fix a point $x$ in $M$ and let $U_{j_1}, \dots, U_{j_k}$ be the elements of the covering $\mathscr{U}$ which contain $x$, so that
\[
\sum_{i=1}^k \varphi_{j_i}(x) = 1,
\]
and $\varphi_j(x)=0$ if $j$ is not in $\{j_1,\dots,j_k\}$.
 By Property (\ref{uno}), there is a point $y$ in $M$ such that
\begin{equation}
\label{due}
x\in U_{j_1} \cup \dots \cup U_{j_k} \subset B_{r(y)}(y).
\end{equation}
Then $\|x-x_{j_i}\|<2 r(y)$ for every $i=1,\dots,k$, and the norm of the vector
\[
h(t,x) - \sum_{i=1}^k \varphi_{j_i}(x) \bigl( t g(x_{j_i}) + (1-t) g(x) \bigr) = t \sum_{i=1}^k \varphi_{j_i}(x) \int_0^1 A\bigl( x_{j_i} + s(x-x_{j_i}) \bigr) (x-x_{j_i})\, ds 
\]
does not exceed the quantity
\[
2 r(y) \sup_{z\in B_{r(y)}(y)} \|A(z)\|_{L(\mathbb{H})},
\]
which by (e) is smaller than $\delta_0(y)/2$. On the other hand, (c) implies that
\[
\Bigl\| g(y) - \sum_{i=1}^k \varphi_{j_i}(x) \bigl( t g(x_{j_i}) + (1-t) g(x) \bigr) \Bigr\| < \frac{\delta_0(y)}{2}.
\]
Therefore,
\[
\|h(t,x)-g(y)\| < \delta_0(y),
\]
and by the definition of $\delta_0$ we conclude that $h$ and, a fortiori, $\overline{f}$ take values into $N$.

If, moreover, $x$ belongs to $U$, then (b) guarantees that $B_{r(y)}(y)$ is contained in $V$, so by (\ref{due}), (i) and (ii),
\begin{eqnarray*}
h(t,x) & = & t \sum_{j=1}^k \varphi_{j_i}(x) \Bigl( f(x_{j_i}) + \int_0^1 df\bigl(x_{j_i} + s(x-x_{j_i}) \bigr) (x-x_{j_i})\, ds \Bigr) + (1-t) f(x) \\
& = & t \sum_{j=1}^k \varphi_{j_i}(x) \Bigl( f(x_{j_i}) + \int_0^1\frac{d}{ds} f\bigl(x_{j_i} + s(x-x_{j_i}) \bigr) \, ds \Bigr) + (1-t) f(x) \\
& = & t \sum_{j=1}^k  \varphi_{j_i}(x) f(x) + (1-t) f(x) = f(x),
\end{eqnarray*}
hence $\overline{f}|_U=f|_U$ and $h$ is a homotopy between $\overline{f}$ and $g$ relative $U$, proving (i').

The differential of $\overline{f}$ at $x$ is the operator
\[
d\overline{f}(x) = A_1(x) + A_2(x) + A_3(x),
\]
where 
\begin{eqnarray*}
A_1(x) u & := & \sum_{j=1}^k d\varphi_{j_i}(x) [u] \Bigl( g(x_{j_i}) + \int_0^1 A \bigl(x_{j_i} + s(x-x_{j_i}) \bigr) (x-x_{j_i})\, ds \Bigr), \quad \forall u\in \mathbb{H}, \\
A_2(x) & := & \sum_{i=1}^k \varphi_{j_i}(x) \int_0^1 A \bigl(x_{j_i} + s(x-x_{j_i}) \bigr) \, ds, \\
A_3(x) u & := & \sum_{i=1}^k \varphi_{j_i}(x) \int_0^1 s \, dA\bigl( x_{j_i} + s(x-x_{j_i}) \bigr) [u] (x-x_{j_i})\, ds, \quad \forall u\in \mathbb{H}.
\end{eqnarray*}
The operator $A_1(x)$ has rank at most $k$. By (\ref{due}) and (d), 
\begin{equation}
\label{tre}
\|A_2(x)-A(y)\|_{L(\mathbb{H})} < \frac{\delta_1(y)}{2}.
\end{equation}
By (\ref{due}) and (f),
\begin{equation}
\label{quattro}
\|A_3(x)\|_{L(\mathbb{H})} \leq 2 r(y) \sup_{z\in B_{r(y)}(y)} \|dA(z)\|_{L(\mathbb{H},L(\mathbb{H}))} \int_0^1 s\, ds < \frac{\delta_1(y)}{2}.
\end{equation} 
By (\ref{tre}) and (\ref{quattro}), 
\[
\|A_2(x)+A_3(x)-A(y)\|_{L(\mathbb{H})} < \delta_1(y),
\]
so by the definition of $\delta_1$, $A_2(x)+A_3(x)$ belongs to $\Phi_n(\mathbb{H})$. Since $A_1(x)$ has finite rank, $d\overline{f}(x)$ belongs to $\Phi_n(\mathbb{H})$ too. Hence, $\overline{f}$ is a Fredholm map of index $n$.

By using again (\ref{due}) and (d), 
\[
\| t A_2(x) + (1-t) A(x) - A(y) \|_{L(\mathbb{H})} < \frac{\delta_1(y)}{2}, \quad \forall t\in [0,1].
\]
Together with (\ref{quattro}) and the fact that $A_1(x)$ has finite rank, this implies that the homotopy
\[
H(t,x) := t d\overline{f}(x) + (1-t) A(x) = t A_1(x) + \bigl( t A_2(x) + (1-t) A(x) \bigr) + t A_3(x)
\]
takes values into $\Phi_n(\mathbb{H})$. If, moreover, $x$ belongs to $U$, then $A(x)=df(x)=d\overline{f}(x)$. Therefore,
\[
H(t,x) = df(x), \quad \forall x\in U,
\]
and $H$ is a homotopy relative $U$, proving (ii').

\section{Proper Extension of Fredholm maps}
\label{app:properextention}

The aim of this appendix is to prove Theorem \ref{thm:totalextension} on the extension of proper Fredholm maps. The proof follow closely the proof of \cite[Lemma 4.2]{et70}, which has slightly stronger assumptions.

We start by recalling the following result of Bessaga \cite{bes66}. See also \cite[III.6]{bp75}.

\begin{theorem}
\label{thm:bessaga} There exists a diffeomorphism $d_1:\mH\rightarrow
\mH\setminus \{0\}$ which is the identity outside the ball of radius $1/2$.
\end{theorem}

Let $B_r(x)$ be the open ball in $\mH$ with center $x$ and radius $r$, and let $S:= \partial B_1(0)$ 
be the unit sphere in $\mH$. A famous corollary of Bessaga's theorem is that $S$ is diffeomorphic to $\mH$. 

\begin{corollary} 
\label{cor:spherehilbertdiffeo}
There exists a diffeomorphism $d_2:S\rightarrow \mH$.
\end{corollary}

\begin{proof}
Let $x_0\in S$. We first construct a diffeomorphism $S\cong S\setminus \{x_0\}$ by applying $d_1$ in a chart around $x_0$. The stereographic projection $\phi:S\setminus \{x_0\}\rightarrow x_0^\perp$ defined by
\[
\phi(x)=\frac{x-\langle x,x_0\rangle x_0}{1-\langle x,x_0\rangle}
\]
is a diffeomorphism onto the hyperplane $x_0^\perp$ orthogonal to $x_0$, which in turn is diffeomorphic to $\mH$. By composition, we obtain a diffeomorphism $d_2:S\rightarrow \mH$.
\end{proof}

Bessaga's Theorem allows us to construct various other useful maps, for example an involution on $\mH$ which exchanges the interior with the exterior of $S$:

\begin{corollary}
\label{cor:J}
There is a diffeomorphism $J:\mH\rightarrow\mH$ such that
\[
J\bigr|_S=\id,\qquad J^2=\id, \qquad\text{and}\qquad J \bigl( B_1(0) \bigr) = \mH \setminus \overline{B}_1(0).
\]
The map $J$ is Fredholm homotopic to $\id$ relative $S$.
\end{corollary}

\begin{proof}
Consider the diffeomorphism $j:\mH\setminus \{0\}\rightarrow \mH\setminus\{0\}$ given by $j(x)=x/\|x\|^2$. Then $j^2=\id$ and the diffeomorphism $J:\mH\rightarrow \mH$ defined by $J=d_1^{-1}\circ j \circ d_1$ also satisfies $J^2=\id$. Since $d_1$ maps the interior of the unit ball to itself, and the complement of the unit ball to itself, we have $J \bigl( B_1(0) \bigr) = \mH \setminus \overline{B}_1(0)$. 

For every $t\in [0,1]$ let $j_t:\mH\setminus \{0\}\rightarrow \mH\setminus \{0\}$ be defined by $j_t(x)=\lambda_t(x)x$, with the smooth function $\lambda_t:\mH\setminus\{0\}\rightarrow \mR$ defined by
\[
\lambda_t(x)=1-t+\frac{t}{\norm{x}^2}.
\]
Since $\lambda_t(x)>0$ for every $t\in [0,1]$ and $x\neq 0$, this map is well defined and Fredholm: indeed, its differential 
\[
dj_t(x)=x \, d\lambda_t(x)+\lambda_t(x)\id
\] 
is a Fredholm operator, being a perturbation of rank at most one of the isomorphism $\lambda_t(x)\id$.
 Then $J_t=d_1^{-1}\circ j_t \circ d_1$ is the required Fredholm homotopy. For each $x\in S$ we see that $j_t(x)=x$, so $J_t$ is a homotopy relative $S$.
\end{proof}

We will need the following extension result for proper Fredholm maps into the unit sphere:

\begin{lemma}
\label{lem:propfredholm}
Let $M$ be a Hilbert manifold and $f:M\rightarrow S$ a Fredholm map into the infinite dimensional sphere of $\mH$. Let $A\subset M$ be a closed set and $ V$ a neighborhood of $A$ such that $f\bigr|_{\overline{V}}$ is proper ($A$ and $V$ might be empty). Then there exists a smooth function $\lambda:M\rightarrow [1,+\infty)$ such that $A\subset \lambda^{-1}(1)\subset V$ and the map $g:M\rightarrow \mH$ defined by
\[
g(x):=\lambda(x)f(x)
\]
is proper and Fredholm of index $\mathrm{ind} \, f -1$.
\end{lemma}

\begin{proof}
Since Fredholm maps are locally proper (see e.g.\ \cite[Theorem 1.6]{sma65}), every $x\in M$ has a neighborhood $U$ such that $f\bigr|_{\overline{U}}$ is proper. Since $M$ is paracompact we can find a countable locally finite open cover $\mathscr{U} = \{U_j\}_{j\in\mN}$ of $M\setminus A$ such that $f\bigr|_{\overline{U_j}}$ is proper. Set $U_0:=V$ and consider the covering $\mathscr{U} \cup \{U_0\}$ of $M$. Let $\{\phi_j\}_{j\in \mN_0}$, $\mN_0 = \{0\} \cup \mN$, be a smooth partition of unity subordinated to the latter covering such that
\begin{equation}
\label{doveA}
A \subset \phi_0^{-1}(1).
\end{equation}
Consider the function 
\[
\lambda:M \rightarrow [1,+\infty), \qquad \lambda :=\sum_{j=0}^{\infty} 2^j\phi_j,
\]
and set $g:=\lambda f$. By (\ref{doveA}) we have $A\subset \lambda^{-1}(1)$. If $x\notin V=U_0$ then $\phi_0(x)=0$ and hence $\lambda(x)\geq 2$; this implies that $\lambda^{-1}(1) \subset V$. For every $n\in \mN$ the set
\[
\lambda^{-1}([1,2^n])
\]
is contained in the union $U_0 \cup U_1 \cup \dots \cup U_n$, because on the complement of this union all $\phi_j$'s with $0\leq j\leq n$ vanish and hence $\lambda \geq 2^{n+1}$. The restrictions $g|_{\overline U_j}$ are proper, because $g(M) \subset \mH \setminus B_1(0)$ and for every $K\subset \mH \setminus B_1(0)$ there holds
\[
(g|_{\overline {U_j}})^{-1}(K)\subset(f\bigr|_{{\overline {U_j}}})^{-1}(\pi(K)),
\]
where $\pi:\mH\setminus \{0\}\rightarrow S$ is the radial projection onto the sphere. If $K$ is compact it is in particular closed and the set $(g|_{\overline {U_j}})^{-1}(K)$ is a closed subset of a compact set as $f$ is proper when restricted to $\overline{U_j}$, thus $(g|_{\overline {U_j}})^{-1}(K)$ is compact.

It follows that the restriction of $g$ to any finite union of the $\overline{U_j}$'s is proper, and from the fact that for every $R\geq 1$ the set
\[
\lambda^{-1}\bigl([1,R]\bigr)
\]
is contained in finitely many of the $\overline{U_j}$'s we conclude that $g$ is proper. At a point $x\in M$ we have
\[
dg(x)=f(x)\,d\lambda(x)+\lambda(x)\, df(x).
\]
The map $f(x)d\lambda(x)$ has rank at most one, and since $\lambda(x)$ does not vanish $\lambda(x)df(x): T_x M \rightarrow \mH$ is a Fredholm operator whose index equals the index of $df(x): T_x M \rightarrow T_{f(x)} S$ minus 1. Hence $g$ is a Fredholm map with index $\mathrm{ind}\, f-1$.
\end{proof}

The main step in the proof of Theorem \ref{thm:totalextension} is contained in the following result. The difference with respect to \cite[Lemma 4.2]{et70} is that the point $z$ is not required to belong to $f(U)$, an assumption which turns out to be superfluous.

\begin{lemma}
\label{lem:propextension}
Let $M$ be a Hilbert manifold and $U,V$ be open subsets of $M$ such that $\overline U\subset V$. Let $f:M\rightarrow \mH$ be a Fredholm map of index $n$ such that $f\bigr|_{\overline{V}}$ is proper. Moreover, assume that there exists a point $z$ in $\mH \setminus f(\partial U)$. Then there exists a proper Fredholm map $\overline f:M\rightarrow \mH$ of index $n$ such that $\overline f|_U = f|_U$ and
\begin{enumerate}[(a)]
\item $\overline f$ and $f$  are Fredholm homotopic relative $U$;
\item there exists a neighborhood $Z\subset \mH$ of $z$ such that $\overline f^{-1}(Z)=f^{-1}(Z)\cap U$. 
\end{enumerate}
\end{lemma}
\begin{proof}
By Corollary \ref{cor:spherehilbertdiffeo} we can replace the Hilbert space $\mH$ by the Hilbert sphere $S$ in this statement: $f$ is a Fredholm map of index $n$ from $M$ to $S$ such that $f|_{\overline{V}}$ is proper, $z$ is a point in $S\setminus f(\partial U)$ and we must find a proper Fredholm map $\overline{f} : M \rightarrow S$ of the same index such that $\overline{f}|_U = f|_U$ and (a) and (b) hold.

A proper map into a metric space is closed and hence $f(\partial U)$ is closed. Let $Y\subset S$ be an open neighborhood of $f(\partial U)$ such that $z$ does not belong to $\overline{Y}$. Let $V_1$ be an open neighborhood of $\partial U$ such that $f(V_1)\subset Y$ and $\overline{V_1} \subset V$. Then the open set $V_0:= U\cup V_1$ satisfies
\[
\overline{U} \subset V_0 \subset \overline{V_0} \subset V
\]
and is such that $z$ does not belong to $f(\overline{V_0} \setminus U)$. Let $A$ be a closed neighborhood of $\overline{U}$ contained in $V_0$.  By Lemma \ref{lem:propfredholm} there exists a smooth function $\lambda: M \rightarrow [1,+\infty)$ such that
\[
 A \subset \lambda^{-1}(1) \subset V_0
\]
and $g_1:= \lambda f: M \rightarrow \mH$ is a proper Fredholm map. The index of $g_1$ is $n-1$ and its image is contained in $\mH \setminus B_1(0)$. Let $i: S \hookrightarrow \mH$ denote the inclusion mapping. Then $g_1$ is Fredholm homotopic to $i\circ f$ relative $U$ via the homotopy
\[
\tilde{h}_t := t g_1 + (1-t) i\circ f = (t\lambda + 1- t) i\circ f, \qquad t\in [0,1].
\]
Indeed, the formula
\[
d \tilde{h}_t(x) = t f(x) d\lambda(x) + \bigl( t \lambda(x) + 1-t \bigr) df(x)
\]
shows that the operator $d\tilde{h}_t(x)$ is Fredholm, being a finite rank perturbation of the Fredholm operator $( t \lambda(x) + 1-t) df(x)$. 

By construction, $z$ does not belong to $g_1(M\setminus U)$ and since the latter set is closed, there exists an open ball $B\subset \mH$ containing $z$ such that $B\cap g_1(M\setminus U) = \emptyset$. Set
\[
g_2:= J \circ g_1 : M \rightarrow \mH,
\]
with $J$ as in Corollary \ref{cor:J}. The image of $g_2$ is contained in the closed unit ball of $\mH$. Since $J$ is a diffeomorphism, $g_2$ is Fredholm homotopic to $g_1$ relative $U$ by the corresponding property of $J$ stated in Corollary \ref{cor:J}. Therefore, $g_2$ is Fredholm homotopic to $i\circ f$ relative $U$.

Since $z$ is fixed by $J$, we can find an open ball $B_0 \subset \mH$ containing $z$ and such that $B_0 \subset J(B)$. It follows that $B_0 \cap g_2(M\setminus U)=\emptyset$. Let 
\[
r : \overline{B}_1(0) \setminus \{z\} \rightarrow \overline{B}_1(0) \setminus \{z\}
\]
be the retraction onto $S\setminus \{z\}$ which is defined by mapping each $x\in  \overline{B}_1(0) \setminus \{z\}$ into the intersection of the ray emanating from $z$ in the direction of $x$ and the set $S\setminus \{z\}$. This map is Fredholm and is Fredholm homotopic to the identity mapping on $\overline{B}_1(0) \setminus \{z\}$ relative $S\setminus \{z\}$ by convex interpolation.  The retraction $r$ does not extend continuously to $z$. However, since $g_2$ maps  a neighborhood of $\partial U$ into $S\setminus \{z\}$, the map
\[
g_3: M \rightarrow \mH, \qquad g_3(x) = \left\{ \begin{array}{ll} r(g_2(x)) & \mbox{if } x\in M \setminus U, \\ g_2(x) & \mbox{if } x\in U, \end{array} \right.
\]
is smooth. The map $g_3$ is Fredholm of index $n-1$ and is proper, because $g_2$ is proper and $r$ is proper on $\overline{B}_1(0) \setminus B_0$. Moreover, the fact that $r$ is Fredholm homotopic to the identity relative $S\setminus \{z\}$ implies that $g_3$ is Fredholm homotopic to $g_2$ relative $U$. Therefore, $g_3$ is Fredholm homotopic to $i\circ f$ relative $U$. Denote by 
\[
h: [0,1] \times M \rightarrow \mH
\]
such a Fredholm homotopy. Since the image of $g_3$ is contained in $S$, we can restrict its codomain and obtain a proper Fredholm map $\overline{f}: M \rightarrow S$ of index $n$. The image of $[0,1]\times U$ by the map $h$ is contained in $S$. Let $\pi : \mH \setminus \{0\} \rightarrow S$ be the radial projection and $d_1: \mH \rightarrow  \mH \setminus \{0\}$ the diffeomorphism of Theorem \ref{cor:spherehilbertdiffeo}. Then 
\[
\pi\circ d_1 \circ h: [0,1]\times M \rightarrow S
\]
is a Fredholm homotopy, relative $U$, between $f = \pi\circ d_1 \circ i \circ f$ and $\overline{f} = \pi \circ d_1 \circ g_3$. Therefore, the proper Fredholm map $\overline{f}$ satisfies $\overline{f}|_U=f|_U$ and (a). It also satisfies (b) with $Z:= B_0 \cap S$.
\end{proof} 

We can finally prove Theorem \ref{thm:totalextension}

\begin{proof}[Proof of Theorem \ref{thm:totalextension}]
By Corollary \ref{fredextH} we can find a Fredholm map $g:M \rightarrow \mH$ such that $g|_V = f|_V$ and
\begin{enumerate}[(a)]
\item $dg: M \rightarrow \Phi_n(\mH)$ is homotopic to $A$ relative $V$.
\end{enumerate}  
Then $g|_{\overline V} = f|_{\overline V}$ is proper and $z$ belongs to $\mH \setminus g(\partial U)$. By Lemma \ref{lem:propextension} applied to $g$ we can find a proper Fredholm map $\overline{f} : M \rightarrow \mH$ which coincides with $g|_U=f|_U$ on $U$ and such that: 
\begin{enumerate}[(a)]
\setcounter{enumi}{1}
\item $\overline{f}$ and $g$ are Fredholm homotopic relative $U$;
\item there exists a neighborhood $Z\subset \mH$ of $z$ such that $\overline{f}^{-1}(Z) = g^{-1}(Z) \cap U = f^{-1}(Z) \cap U$.
\end{enumerate}
Then the proper Fredholm extension $\overline{f}$ of $f|_U$ satisfies (i') by (a) and (b), and (ii') by (c). 
\end{proof}

\section{Weakening the hypotheses.}
\label{weakening}

In many applications to nonlinear partial differential equations, it is useful to weaken some of the assumptions that we made in this article: On the one hand, one would like to work with maps with finite regularity, on the other hand one would like to replace the Hilbert space $\mH$ by more general Banach spaces. In this concluding appendix, we briefly discuss these two issues.

All the notions which we have introduced above make sense when the Fredholm maps are just of class $C^1$. Higher regularity is needed in the proofs when one wants to apply the Sard-Smale theorem, see Theorem \ref{smale}. By taking into account the minimal regularity needed in this theorem, we can make the following observations: Theorems \ref{main1} and \ref{main2} hold for Fredholm maps of class $C^{n+2}$, Theorem \ref{main3} just requires $C^1$ maps, and Theorem \ref{main4} needs $C^2$ regularity.

By using regularization by smooth partitions of unity, it is not difficult to show that any $C^1$ Fredholm map $f: M\rightarrow N$ is Fredholm homotopic to a smooth Fredholm map $\tilde{f}: M \rightarrow N$. Indeed, one can smoothen $f$ and keep the smoothing Fredholm by seeing $N$ as an open subset of $\mH$ (see \cite{ee70} again) and by setting
\[
\tilde{f}(x) := \sum_{n\in \mN} \phi_n(x) \bigl( f(x_j) + df(x_j)[x-x_j] \bigr), \qquad \forall x\in M,
\]
where $\{\phi_n\}_{n\in \mN}$ is a suitable smooth partition of unity on $M$. Using this fact, one can easily show that Theorem \ref{main1} actually holds for Fredholm maps of class $C^1$.

Unfortunately, we do not know whether it is possible to regularize proper Fredholm maps by keeping also the properness, and hence get the validity of Theorems \ref{main2} and \ref{main4} for $C^1$ maps just by regularization. In the literature, various alternative approaches have been developed in order to deal with this difficulty, starting with \cite{sap73} and \cite{isn74}. A particularly interesting approach based on finite dimensional reduction is developed in \cite{zr15}. It produces a degree theory for $C^1$ proper Fredholm maps of index zero between Banach manifolds, and also homotopy  invariants for proper Fredholm maps of positive degree. It would be interesting to compare these invariant with the complete ones presented here.

Now we turn to the question of replacing $\mH$ by a more general Banach space $\mathbb{E}$. Among the properties of the model space that we use, a couple stand out as crucial:

\begin{enumerate}[(i)]
\item $\mathbb{E}$ is a Kuiper space, i.e.\ the general linear group $GL(\mathbb{E})$ is contractible. It is known that not all infinite dimensional Banach spaces are Kuiper. For an overview of Kuiper and non-Kuiper Banach spaces, we refer to \cite{mitjagin}.
\item $\mathbb{E}$ is stable. This means that $\mathbb{E}$ is isomorphic to $\mR \times \mathbb{E}$. Non-stable Banach spaces exist. Gowers \cite{gowers} was the first to construct an example of such a space. 
\item $\mathbb{E}$ is diffeomorphic to its unit sphere. In \cite{aza97} it is shown that any infinite dimensional Banach space with $C^k$ norm is $C^k$  diffeomorphic to its unit sphere.
\item For the explicit computation of the homotopy classes of proper Fredholm maps of non-positive index we use that the homotopy type of the space of Fredholm operators is known, namely it equals $\mathbb{Z}\times BO$. Koschorke \cite{koschorke} has shown that the space of Fredholm operators of index zero of a infinite dimensional separable Kuiper Banach space is homotopy equivalent to $BO$. Moreover, if one further assumes $\mathbb{E}$ to be stable, the homotopy type of the space of all Fredholm operators on $\mathbb{E}$ equals $\mathbb{Z}\times BO$. 
\end{enumerate}
An inspection to our proofs and the above results show that it is possible to replace $\mH$ by a real separable Banach space $\mathbb{E}$ which is Kuiper, stable and whose norm is sufficiently regular. Indeed, when working on Banach spaces, or Banach manifolds, an extra difficulty is the lack of smooth partitions of unity. Due to this fact, the proofs of Theorems \ref{main1}, \ref{main2}, \ref{main3} and \ref{main4} extend to paracompact manifolds modeled on a separable, stable and Kuiper Banach space $\mathbb{E}$ when $\mathbb{E}$ admits a norm of class $C^{n+2}$, respectively $C^{n+2}$, $C^1$ and $C^2$. Lowering the regularity assumptions to $C^1$ as discussed above is hence relevant also for the extension to the Banach setting. We expect Theorems \ref{main1}, \ref{main2}, \ref{main3} and \ref{main4} to hold for Fredholm maps of class $C^1$ on  paracompact manifolds modeled on a separable, stable and Kuiper Banach space $\mathbb{E}$ admitting a norm of class $C^1$.

\bibliographystyle{alpha} 
\bibliography{pontryaginthom}

\end{document}